\documentclass[11pt, british]{amsart}

\usepackage{babel,eucal,url,amssymb,
enumerate,amscd,
}
\usepackage{color}
\usepackage{framed}

 \newtheorem{thm}{Theorem}[section]
 \newtheorem{cor}[thm]{Corollary}
 \newtheorem{lem}[thm]{Lemma}
 \newtheorem{prop}[thm]{Proposition}
 \newtheorem{defn}[thm]{Definition}
 \newtheorem{rem}[thm]{Remark}
 \newtheorem{example}[thm]{Example}
 \numberwithin{equation}{section}

\begin{document}

\title[] {Almost paracontact metric  Walker 3-manifolds}
    \author{Cornelia-Livia Bejan}
 \address{Cornelia-Livia Bejan} 
\curraddr{Department of Mathematics, ''Gh. Asachi'' Technical University of Ia\d si, \\ 
B-dul Carol I, nr. 11, 700506 Ia\d si, Rom\^{a}nia}
 \email{bejanliv@yahoo.com}

    \author{Galia Nakova}
    \address{Galia Nakova} 
\curraddr{St. Cyril and St. Methodius University of Veliko Tarnovo \\ Faculty of Mathematics and Informatics,  Department of Mathematics
\\ 2 Teodosii Tarnovski Str., Veliko Tarnovo 5003,  Bulgaria}
    \email{g.nakova@ts.uni-vt.bg, gnakova@gmail.com}

%






\begin{abstract}
On a Walker 3-manifold $(M,g)$, we construct and study almost paracontact metric structures $(\varphi ,\xi ,\eta ,g)$, only by means of coordinate functions of a unit space-like vector field $\xi $ and the metric $g$. Necessary and sufficient conditions are established for $M$, endowed with these structures,  to belong to one of the classes of almost paracontact metric 3-manifolds, according to the classification of S. Zamkovoy and G. Nakova. We find some classes to which the studied manifolds do not belong.  The almost paracontact metric structure on a Walker 3-manifold examined in \cite{GC2}  is obtained here as a particular case.  Special attention we pay to an $\eta $-Einstein manifold, whose $\xi $-sectional, $\varphi $-sectional and scalar curvatures are investigated. Some local and global examples are constructed throughout.
\end{abstract}

\newcommand{\ie}{i.\,e. }
\newcommand{\g}{\mathfrak{g}}
\newcommand{\D}{\mathcal{D}}
\newcommand{\F}{\mathcal{F}}
\newcommand{\C}{\mathcal{C}}
\newcommand{\W}{\mathcal{W}}
\newcommand{\N}{\mathcal{N}}
\newcommand{\s}{\mathop{s}}

\newcommand{\diag}{\mathrm{diag}}
\newcommand{\End}{\mathrm{End}}
\newcommand{\im}{\mathrm{Im}}
\newcommand{\id}{\mathrm{id}}
\newcommand{\Hom}{\mathrm{Hom}}

\newcommand{\Rad}{\mathrm{Rad}}
\newcommand{\rank}{\mathrm{rank}}
\newcommand{\const}{\mathrm{const}}
\newcommand{\tr}{{\rm tr}}
\newcommand{\ltr}{\mathrm{ltr}}
\newcommand{\codim}{\mathrm{codim}}
\newcommand{\Ker}{\mathrm{Ker}}
\newcommand{\R}{\mathbb{R}}
\newcommand{\E}{\mathbb{E}}
\newcommand{\K}{\mathbb{K}}

\newcommand{\thmref}[1]{Theorem~\ref{#1}}
\newcommand{\propref}[1]{Proposition~\ref{#1}}
\newcommand{\corref}[1]{Corollary~\ref{#1}}
\newcommand{\secref}[1]{\S\ref{#1}}
\newcommand{\lemref}[1]{Lemma~\ref{#1}}
\newcommand{\dfnref}[1]{Definition~\ref{#1}}


\newcommand{\ee}{\end{equation}}
\newcommand{\be}[1]{\begin{equation}\label{#1}}

\maketitle

\section{Introduction}\label{sec-1}
The study of geometric structures on differentiable manifolds is a central theme in modern differential geometry, where additional tensor fields on manifolds enrich their features and open connections with mathematical physics. Investigating pseudo-Riemannian manifolds there naturally arise two important objects - almost paracontact metric structures and Walker manifolds.  In 1950, A. G. Walker introduced  his concept, which plys a crucial role in modern geometry and theoretical physics (in particular in relativity and string theory). Walker structures appeared in the study of PP-waves, supersymmetry and supergravity.
\par
Since physical phenomena are modeled by various tensors, we deal here with
almost paracontact metric structures which are the odd-dimensional counterpart of almost para-Hermitian structures. Early contributions in this direction can be traced to Kaneyuki and his collaborators, who studied paracomplex and paracontact analogues of contact geometry in the 1980s. Later, the framework was systematically developed by researchers such as Zamkovoy (see \cite{Z}) and others \cite{Sahin}, extending Blair's influential work on contact metric geometry to the pseudo-Riemannian and "para" contexts.
\par
On the other hand, Walker manifolds arise naturally in Lorentzian geometry. A Walker manifold is a pseudo-Riemannian manifold that admits a parallel lightlike (null) distribution - that is, a totally lightlike subbundle of the tangent bundle, which remains invariant under parallel transport with respect to the Levi-Civita connection.  The existence of such distributions constrains the holonomy group and has profound consequences for curvature. In dimension three, these manifolds play a particularly important role: they serve as local models for Lorentzian geometries with a parallel null vector field.
\par
The main goal of this paper is to define and study almost paracontact metric structures on Walker 3-manifolds in the general case, as well as to give necessary and sufficient conditions for the resulting manifolds to belong to the classes of  almost paracontact metric 3-manifolds, according to the well known classification of S. Zamkovoy and the second author. 
\par
On a Walker 3-manifold $(M,g)$, a  paracontact metric and a para-Sasakian structure $(\varphi ,\xi ,\eta ,g)$  were considered  in \cite{GC2}, only in the particular case, when the Reeb vector field $\xi $ is an eigenvector of the Ricci operator and  its dual form $\eta $ is a contact form (i.e.  $\eta \wedge d\eta \neq 0$).
\par
To achieve our goal, we provide in Section 2  some basic notions and characterizations  of certain classes of almost paracontact metric 3-manifolds from our basic reference \cite{ZN}.  In Proposition \ref{Proposition 2.2},  we obtain necessary and sufficient conditions for an almost paracontact metric 3-manifold to be (almost) $\alpha $-paracosymplectic, (almost) $\alpha $-para-Kenmotsu or (almost) paracosymplectic.
\par
 Walker 3-manifolds $M$ are described at the beginning of Section 3. On such manifolds, an existence result for almost paracontact metric structures $(\varphi ,\xi ,\eta ,g)$  is obtained (Theorem \ref{Theorem 3.1}). These structures  are determined with respect to the local basis $\left\{\partial _x,\partial _y,\partial _z\right\}$ only by the Reeb vector field $\xi $  and a three-variables function $f(x,y,z)$ on $M$, which characterizes the Lorentzian metric $g$.
\par
In Section 4, we prove a nonexistence result, namely  a Walker 3-manifolds endowed with these almost paracontact metric structures are never para-Sasakian (equivalently, K-paracontact metric) (Theorem \ref{Theorem 4.1}). We give some necessary and sufficient conditions  for an almost paracontact metric Walker 3-manifold to be paracontact metric  (Proposition \ref{Proposition 4.1}). Then we establish some cases in which almost paracontact metric  Walker 3-manifolds  are never paracontact metric. 
Further, on a  Walker 3-manifold, we explicitly find a family of paracontact metric structures  that depends on two smooth functions on the manifold (Theorem \ref{Theorem 4.2}). In the remaining part of this section, we give  necessary and sufficient conditions for the almost paracontact metric Walker 3-manifolds  to fall in one of the following classes: almost $\alpha $-paracosymplectic, almost paracosymplectic, paracosymplectic, normal and $\mathbb{G}_{12}$-manifolds (Theorem \ref{Theorem 4.3}, Theorem \ref{Theorem 4.4}, Theorem \ref{Theorem 4.5}).  We also provide examples of the examined manifolds.
\par
In Section 5 we prove that an almost paracontact metric Walker 3-manifold $(M,\varphi ,\xi ,\eta ,g)$ is $\eta $-Einstein if and only if the Ricci operator $Q$ of $M$ is of Segre type $\{11,1\}$ degenerate case and  $\xi  $ is an eigenvector of $Q$ for the eigenvalue $\lambda _1=0$ (Theorem \ref{Theorem 5.1}). We obtain that  $\varphi $ and $Q$ commute if and only if $M$ is either flat, or $M$ is $\eta $-Einstein (Theorem \ref{Theorem 5.2}). Moreover, we show that an $\eta $-Einstein almost paracontact metric Walker 3-manifold  is of constant non-zero scalar curvature, zero $\xi $-sectional curvature and constant non-zero $\varphi $-sectional curvature, as well as, that $M$ is either  paracosymplectic, or almost paracosymplectic (Theorem \ref{Theorem 5.3}).
\par
In Section 6, we construct some global examples. 
\par
 Throught the paper,  all geometric objects are of class $C^\infty $. 
\section{Almost paracontact metric manifolds}\label{sec-2}
\par
An almost paracontact structure on a  $(2n+1)$-dimensional $C^\infty $-manifold $M$ is a triplet $(\varphi ,\xi ,\eta )$, where   
$\varphi $ is an endomorphism of the tangent bundle $TM$, $\xi $ a global vector field and $\eta $ a 1-form satisfying the following relations  \cite{KW}:
\begin{equation}\label{2.1}
\varphi ^2={\rm Id}-\eta \otimes \xi ;
\end{equation}
\begin{equation}\label{2.2}
\eta (\xi )=1, \quad \varphi \xi =0,
\end{equation}
where {\rm Id} denotes the identity transformation;
\begin{equation}\label{2.3}
\begin{array}{llll}
\text {the restriction}  \, \, \varphi \vert _{\mathbb D} \, \, \text {of} \, \,  \varphi   \, \, \text {on the paracontact distribution } \\
{\mathbb D}={\rm Ker}\, \eta  \, \,
 \text {is an almost paracomplex structure} (\varphi ^2 \vert _{\mathbb D}={\rm Id}) \\
 \text {and the} \, \, 
 \text{eigensubbundles} \, \, {\mathbb D}\, ^+  \text{and} \, \,  
{\mathbb D}\, ^-
\text{corresponding}\, \, \text{to the} \\
 \text{eigenvalues} \, \, 
\text{1 and $-1$} \, \, 
\text{of} \, \,  \varphi |_{\mathbb D}, \, 
\text{have the same dimension}\, \, n.
\end{array}
\end{equation}
\par
As an immediate consequence of \eqref{2.1} and \eqref{2.2} we obtain
\begin{equation}\label{2.4}
\eta \circ \varphi=0.
\end{equation}
Everywhere here we will use the following notations: $C^\infty ( M)$ is  the set of all $C^\infty $ functions on $M$;
$\chi (M)$ is  the set of all $C^\infty $ vector fields on $M$; $\mathcal{D}^r(M)$ is the set of all $C^\infty $ $r$-forms on $M$.
If $M$ is endowed with an almost paracontact structure $(\varphi ,\xi ,\eta )$ and $g$ is 
a pseudo-Riemannian metric such that
\begin{equation}\label{2.5}
g(\varphi X,\varphi Y)=-g(X,Y)+\eta (X)\eta (Y), \quad X, Y \in \chi (M),
\end{equation}
then $(M,\varphi ,\xi ,\eta ,g)$ is called {\it an almost paracontact metric manifold}  equipped with an almost paracontact metric structure $(\varphi ,\xi ,\eta ,g)$. The metric $g$ is called compatible metric and it is necessarily of signature $(n+1,n)$(see \cite {Z}).
\par
Setting $Y=\xi $ in \eqref{2.5}, we get
\begin{equation}\label{2.6}
\eta (X)=g(X,\xi ).
\end{equation}
Also, from \eqref{2.5}, by using \eqref{2.4} and \eqref{2.6}, we have
\begin{equation}\label{2.7}
g(\varphi X,Y)=-g(X,\varphi Y).
\end{equation}
\begin{rem}\label{Remark 2.1}
If $(M,\varphi ,\xi ,\eta ,g)$ is an almost paracontact metric manifold, then from \eqref{2.7} it follows that the eigensubbundles ${\mathbb D}\, ^+$ and ${\mathbb D}\, ^-$ of the paracontact distribution ${\mathbb D}$ are maximal totally isotropic with respect to the restriction of $g$ on ${\mathbb D}$  \cite[Remark, p. 84]{KK}. Hence, ${\rm dim}\, {\mathbb D}\, ^+={\rm dim}\, {\mathbb D}\, ^-=n$.
\end{rem}
\par
Further, by the fundamental 2-form $\phi(X,Y)=g(\varphi X,Y)$ on an almost paracontact metric manifold $(M,\varphi ,\xi ,\eta ,g)$ we define  
the tensor field $F$ by:  
\[
F(X,Y,Z)=(\nabla _{X}\phi )(Y,Z)=g((\nabla _{X}\varphi )Y,Z) , \, \, \, X,Y,Z\in \chi (M),
\]
where $\nabla $ is the Levi-Civita connection on $M$.
It has the properties:
\begin{equation*}
\begin{array}{ll}
F(X,Y,Z)=-F(X,Z,Y), \\
F(X,\varphi Y,\varphi Z)=F(X,Y,Z)+\eta(Y)F(X,Z,\xi)-\eta(Z)F(X,Y,\xi).
\end{array}
\end{equation*}
The following 1-forms are associated with $F$:
\begin{equation}\label{2.8}
\theta _F(X)=g^{ij}F(e_i,e_j,X), \qquad
\theta \, ^*_F(X)=g^{ij}F(e_i,\varphi\,  e_j,X), 
\end{equation}
where $\{e_i,  i=1,\ldots,2n+1\}$ is a local basis of $TM$, and $(g^{ij})$ is the inverse matrix of $(g_{ij})$, with $g_{ij}=g(e_i,e_j)$.
In the next lemma $\nabla \eta$, $d\eta $, 
and $d\phi $ are expressed in terms of the structure tensor $F$.
\begin{lem}\cite{ZN}\label{Lemma 2.1}
For arbitrary $X, Y, Z \in \chi (M)$ we have:
\begin{equation}\label{2.9}
(\nabla _X\eta)Y=g(\nabla _X\xi ,Y)=-F(X,\varphi Y,\xi );
\end{equation}
\begin{equation}\label{2.10}
\begin{array}{ll}
d\eta (X,Y)=\frac{1}{2}\left((\nabla _X\eta)Y-(\nabla _Y\eta)X\right)\\ \\
\qquad \qquad =\frac{1}{2}(-F(X,\varphi Y,\xi )+F(Y,\varphi X,\xi ));
\end{array}
\end{equation}
\begin{equation}\label{2.12}
d\phi (X,Y,Z)=F(X,Y,Z)+F(Y,Z,X) +F(Z,X,Y). 
\end{equation}
\end{lem}
\par
For later use we recall some results obtained in \cite{ZN}, where a classification of the almost paracontact metric manifolds with respect to the tensor $F$ is obtained.
\par
Let $\mathcal{F}$ be the subspace of the space $\otimes ^0_3T_pM$ of the tensors of type 
$(0,3)$ over $T_pM$, defined by
\begin{equation*}
\begin{array}{lr}
\mathcal{F}=\{F\in \otimes ^0_3T_pM : F(x,y,z)=-F(x,z,y)=F(x,\varphi \, y,\varphi \, z)\\
\qquad \qquad \qquad \qquad \qquad \qquad  \quad -\eta (y)F(x,z,\xi ) 
+\eta (z)F(x,y,\xi )\}.
\end{array}
\end{equation*}
In \cite{ZN}, the space $\mathcal{F}$ has been decomposed into 12 mutually orthogonal and invariant subspaces $\mathbb{G}_i$ 
$(i=1,\ldots ,12)$ under the action of the structure group $\mathbb {U}^\pi (n)\times \{1\}$, where $\mathbb {U}^\pi (n)$ is the  paraunitary group. Based on this decomposition, 12 basic classes of almost paracontact metric manifolds $\mathbb{G}_i$ 
$(i=1,\ldots ,12)$ with respect to the tensor field $F$ are obtained  and their characteristic conditions are given.
\par
An almost paracontact metric manifold $M$ is said to be in the class $\mathbb{G}_i$ $(i=1,\ldots ,12)$ (or $\mathbb{G}_i$-manifold) if at each $p\in M$ the tensor $F$ of $M$ belongs to the subspace $\mathbb{G}_i$.

The decomposition of $\mathcal{F}$ in to a  direct sum  of the subspaces 
$\mathbb{G}_i$  $(i=1,\ldots ,12)$ implies that every $F\in \mathcal{F}$ has a unique representation in the form  $F(x,y,z)=
\sum \limits _{i=1}^{12} F^i(x,y,z)$, where $F^i\in \mathbb{G}_i$. Hence, an almost paracontact metric manifold $(M,\varphi ,\xi ,\eta ,g )$ belongs to a direct sum of two or more basic classes, i.e. $M\in \mathbb{G}_i\oplus \mathbb{G}_j\oplus \ldots $, if and only if the structure tensor $F$ on $M$ is the sum $F =F^i+F^j+\ldots $, where $F^i$, $F^j$, $\ldots $ are the projections of $F$  \cite{ZN} on the classes 
$\mathbb{G}_i, \mathbb{G}_j, \ldots $, respectively. Moreover, if $M\in \mathbb{G}_i\oplus \mathbb{G}_j\oplus \ldots $, by using \eqref{2.8},
\eqref{2.10} and \eqref{2.12}, we obtain that 
\begin{equation*}
\begin{array}{ll}
\theta _F(X)=\theta _{F^i}(X)+\theta _{F^j}(X)+\ldots ,  \quad \theta ^*_F(X)=\theta ^*_{F^i}(X)+\theta^* _{F^j}(X)+\ldots , \\ \\
 (d\eta )_F=(d\eta )_{F^i}+(d\eta )_{F^j}+\ldots , \quad (d\phi )_F=(d\phi )_{F^i}+(d\phi )_{F^j}+\ldots .
\end{array}
\end{equation*}
\begin{prop}\cite{ZN}\label {Proposition 2.1}
The 3-dimensional almost paracontact metric manifolds belong to the classes $\mathbb{G}_5$,
$\mathbb{G}_6$, $\mathbb{G}_{10}$, $\mathbb{G}_{12}$ and to the classes which are their direct sums.
\end{prop}
The classes $\mathbb{G}_5$, $\mathbb{G}_6$,  $\mathbb{G}_{10}$ and $\mathbb{G}_{12}$ are determined by the conditions: 
\begin{equation}\label{2.13}
\begin{array}{l}
\mathbb{G}_5 : F(X,Y,Z)=\displaystyle{\frac{\theta _F(\xi)}{2n}}\{\eta(Y)g(\varphi X,\varphi Z)-\eta(Z)g(\varphi X,\varphi  Y)\},
\end{array}
\end{equation}
\begin{equation}\label{2.14}
\begin{array}{l}
\mathbb{G}_6 : F(X,Y,Z)=-\displaystyle{\frac{\theta \, ^*_F(\xi)}{2n}}\{\eta(Y)g(X,\varphi  Z)-\eta(Z)g(X,\varphi Y)\},
\end{array}
\end{equation}
\begin{equation}\label{2.15}
\begin{array}{ll}
\mathbb{G}_{10} : F(X,Y,Z)=-\eta(Y)F(X,Z,\xi )+\eta(Z)F(X,Y,\xi ), \\ \\
\qquad \,F(X,Y,\xi )=F(Y,X,\xi )=F(\varphi X,\varphi Y,\xi ),
\end{array}
\end{equation}
\begin{equation}\label{2.16}
\begin{array}{l}
\mathbb{G}_{12} : F(X,Y,Z)=\eta(X)\left\{\eta(Y)F(\xi ,\xi ,Z)-\eta(Z)F(\xi ,\xi ,Y)\right\} .
\end{array}
\end{equation}
\par
The special class 
$\mathbb{G}_0$ is the intersection of all basic classes and it is determined by the condition $F(X,Y,Z)=0$.
From $F=0$ and \eqref{2.9} it follows that $\mathbb{G}_0$ is the class of the almost paracontact metric manifolds with parallel structures, i.e. $\nabla \, \phi =\nabla \varphi =\nabla \, \xi =\nabla \eta =\nabla g=0$.
\par
We also note the important subclass $\overline {\mathbb{G}}_5$ of the class $\mathbb{G}_5$, which consists of all $(2n+1)$-dimensional almost paracontact metric manifolds belonging to $\mathbb{G}_5$, such that $\theta _F(\xi )=2n$ (see  \cite{ZN}). The characteristic condition of $\overline{\mathbb{G}}_5$ is:
\begin{equation}\label{2.17}
\overline{\mathbb{G}}_5: F(X,Y,Z)=\eta(Y)g(\varphi X,\varphi Z)-\eta(Z)g(\varphi X,\varphi Y) .
\end{equation}
By using \eqref{2.8},
\eqref{2.10}, \eqref{2.12} and \eqref{2.13} $\div $ \eqref{2.16}, after some calculations, we obtain
\begin{lem}\label{Lemma 2.2}
Let $(M,\varphi ,\xi ,\eta ,g)$ be a $(2n+1)$-dimensional almost paracontact metric manifold. \\
(i) If  $M\in \mathbb{G}_i$, then $\theta _{F^i}(\xi )=0$, $i=6,10,12$; \\
(ii) If  $M\in \mathbb{G}_i$, then $\theta ^*_{F^i}(\xi )=0$, $i=5,10,12$; \\
(iii) If  $M\in \mathbb{G}_i$, then $d\eta =0$, $i=6,10$; \\
(iv) If $M\in \mathbb{G}_5$, then $d\eta (X,Y)=\frac{\theta _F(\xi )}{2n}g(\varphi X,Y)$; \\
(v) If $M\in \mathbb{G}_{12}$, then $d\eta (X,Y)=\frac{1}{2}\left(\eta(X)F(\xi ,\xi ,\varphi Y)-
\eta(Y)F(\xi ,\xi ,\varphi X)\right)$; \\
(vi) If  $M\in \mathbb{G}_i$, then $d\phi =0$, $i=5,10,12$; \\
(vii) If $M\in \mathbb{G}_6$, then 
\begin{equation*}
\begin{array}{ll}
d\phi (X,Y,Z)=\displaystyle-\frac{\theta ^*_F(\xi )}{n}\{\eta (Y)g(X,\varphi Z)-\eta (Z)g(X,\varphi Y)-\eta (X)g(Y,\varphi Z)\}\\ \\
\qquad \qquad \quad   =\displaystyle-\frac{\theta ^*_F(\xi )}{n}(\eta \wedge \phi )(X,Y,Z).
\end{array}
\end{equation*}
\end{lem}
An almost paracontact metric manifold is called (see \cite{KW, Z})
\begin{itemize}
\item {\it normal} if $N(X,Y)-2{\rm d}\eta (X,Y)\xi = 0$, where
$N(X,Y)=\varphi ^2[X,Y]+[\varphi X,\varphi Y]- \varphi [\varphi X,Y]-\varphi [X,\varphi Y]$
is the Nijenhuis torsion tensor of $\varphi $;
\item {\it paracontact metric} if $\phi =d\eta$;
\item {\it para-Sasakian} if it is normal and paracontact metric;
\item {\it K-paracontact} if it is paracontact and $\xi $ is Killing vector field;
\item {\it quasi-para-Sasakian} if it is normal and $d\phi =0$.
\end{itemize}
The classes of normal, paracontact metric, para-Sasakian, K-paracontact and quasi-para-Sasakian manifolds are determined in \cite{ZN} . These classes for 3-dimensional almost paracontact metric manifolds are given in the following: 
\begin{thm}\cite{ZN}\label {Theorem 2.1}
(a) The classes of the 3-dimensional normal almost paracontact metric manifolds are $\mathbb{G}_5$, $\mathbb{G}_6$ and $\mathbb{G}_5\oplus \mathbb{G}_6$; \\
(b) The classes of the 3-dimensional paracontact metric manifolds are $\overline {\mathbb{G}}_5$ and
$\overline {\mathbb{G}}_5\oplus \mathbb{G}_{10}$; \\
(c) The class of the 3-dimensional para-Sasakian manifolds is $\overline {\mathbb{G}}_5$; \\
(d) The class of the 3-dimensional K-paracontact metric manifolds is $\overline {\mathbb{G}}_5$; \\
(e)  The class of the 3-dimensional quasi-para-Sasakian manifolds is $\mathbb{G}_5$.
\end{thm} 
\begin{defn}\cite{EDM}\label{Definition 2.1}
An almost paracontact metric manifold $(M,\varphi ,\xi ,\eta ,g)$ is said to be an almost $\alpha $-paracosymplectic manifold if
\[
d\eta =0, \qquad d\phi =2\alpha \eta \wedge \phi ,
\]
where $\alpha $ may be a constant or function on $M$. For a particular choices of the function $\alpha $ we have the following subclasses:
\begin{itemize}
\item almost $\alpha $-para-Kenmotsu manifolds, when $\alpha =const \neq 0$,
\item almost paracosymplectic manifolds, when $\alpha =0$.
\end{itemize}
If additionally the normality condition is fulfilled, then manifolds are called $\alpha $-paracosymplectic,
$\alpha $-para-Kenmotsu or paracosymplectic, respectively.
\end{defn}
Next, we determine the classes in the classification given in \cite{ZN} to which belong  the 3-dimensional manifolds from Definition \ref{Definition 2.1}.
\begin{prop}\label{Proposition 2.2}
Let $(M,\varphi ,\xi ,\eta ,g)$ be a 3-dimensional almost paracontact metric manifold. Then we have:
\par
(i) $M$ is almost $\alpha $-paracosymplectic (respectively $\alpha $-paracosymplectic) if and only if  $M\in \mathbb{G}_6\oplus \mathbb{G}_{10}$ but $M\notin \mathbb{G}_{10}$  (respectively $M\in \mathbb{G}_6$) and $\theta ^*(\xi )$ is a function; 
\par
(ii) $M$ is almost $\alpha $-para-Kenmotsu (respectively $\alpha $-para-Kenmotsu) if and only if  $M\in \mathbb{G}_6\oplus \mathbb{G}_{10}$  but $M\notin \mathbb{G}_{10}$ (respectively $M\in \mathbb{G}_6$) and $\theta ^*(\xi )=const$; 
\par
(iii) $M$ is almost paracosymplectic (respectively paracosymplectic) if and only if $M\in \mathbb{G}_{10}$  (respectively $M\in \mathbb{G}_0$).
\end{prop}
\begin{proof}
First, we show that $d\eta =0$ if and only if $M$ belongs to the classes   $\mathbb{G}_6$, $\mathbb{G}_{10}$ and
$\mathbb{G}_6\oplus \mathbb{G}_{10}$.  
Let us assume that $d\eta =0$. From \eqref{2.10} it follows $F(\xi ,\xi ,Y)=0$, which means that $M\notin \mathbb{G}_{12}$. 
Also, by Lemma \ref{Lemma 2.2}, $M\notin \mathbb{G}_{5}$. Hence, $M$ belongs to $\mathbb{G}_6$, $\mathbb {G}_{10}$ and
$\mathbb{G}_6\oplus \mathbb{G}_{10}$. The converse goes straightly. 
\par
According to Lemma \ref{Lemma 2.2}, in the classes $\mathbb {G}_6$ and $ \mathbb {G}_6\oplus \mathbb {G}_{10}$ we have $d\phi =-\theta ^*_F(\xi )(\eta \wedge \phi )$ and $d\phi =0$ in $\mathbb {G}_{10}$, $\mathbb {G}_0$.
\par
Finally, taking into account the above and assertion $(a)$ in Theorem \ref{Theorem 2.1}, it is easy to see that  $(i)$, $(ii)$, and $(iii)$ hold good.
\end{proof}
\section{Almost paracontact metric structures  on a Walker 3-manifold}\label{sec-3}
A Walker 3-manifold is a Lorentzian manifold $(M,g)$, admitting a parallel lightlike (null) distribution $D$ of rank 1. 
Parallel means that for each $X\in \Gamma (D)$ one has $\nabla _YX\in \Gamma (D)$, $\forall \, Y\in  \chi (M)$ and  lightlike  means that $g(X,X)=0$, $\forall \, X\in \Gamma (D)$. 
\par
It is known  \cite{W} that on a  Walker 3-manifold $(M,g)$, around any point, there exist local coordinates $(x,y,z)$, such that with respect to the local basis $\left\{\partial _x,\partial _y,\partial _z\right\}$ the Lorentzian metric $g$ has the following matrix:
\begin{equation}\label{3.1}
g=
\left(\begin {matrix}0 & 0 &  1 \cr  0  &  \epsilon  & 0   \cr 1 & 0 &  f(x,y,z)  
\end {matrix}\right) ,
\end{equation}
where $\epsilon =\pm 1$ and $f(x,y,z)\in C^\infty ( M)$.  We stress that both $D$ and $g$ are geometric structures globally defined on $M$,  while the expression \eqref{3.1} of $g$ is given in some canonical  local charts. Here $D$ is spanned by $\partial _x$.
When the lightlike vector field $\partial _x$ is parallel (i.e. $\nabla _X\partial _x=0$ for each $X\in \chi (M)$), then the Walker manifold 
$(M,g)$ is called {\it strictly Walker manifold}. It is characterized by the condition $f_x=0$ (see \cite{W}).
\par
If $\nabla $ denotes the Levi-Civita connection of $(M,g)$, then the non-zero components of it are given in \cite{CRV} by

\begin{equation}\label{3.2}
\begin{array}{ll}
\nabla _{\partial _x}\partial _z=\frac{1}{2}f_{x}\partial _x , \qquad  \nabla _{\partial _y}\partial _z=\frac{1}{2}f_{y}\partial _x , \\ \\
\nabla _{\partial _z}\partial _z=\frac{1}{2}(ff_{x}+f_z)\partial _x-\frac{1}{2}f_{y}\partial _y-\frac{1}{2}f_{x}\partial _z . 
\end{array}
\end{equation}
\par
Our goal in this section is to define some almost paracontact metric structures on $(M,g)$. First we give a lemma that we will use for the proof of the main result.
\begin{lem}\label{Lemma 3.1}
The structure $(\varphi ,\xi ,\eta ,g)$ on a $(2n+1)$-dimensional  manifold $M$ is almost paracontact metric if and only if  \eqref{2.1}, \eqref{2.6}, \eqref{2.7} and 
\begin{equation}\label{3.3}
g(\xi ,\xi )=1 
\end{equation}
hold.
\end{lem}
\begin{proof}
If $(\varphi ,\xi ,\eta ,g)$ is an almost paracontact metric structure, then \eqref{2.1}, \eqref{2.2}, \eqref{2.3},
\eqref{2.5} are fulfilled and we showed that \eqref{2.6} and \eqref{2.7} are consequences of them. By using $\eta (\xi )=1$ and \eqref{2.6} we obtain \eqref{3.3}.
\par
Conversely, let suppose that \eqref{2.1}, \eqref{2.6}, \eqref{2.7} and \eqref{3.3} hold. Taking into account Remark \ref{Remark 2.1}, 
to prove that $(\varphi ,\xi ,\eta ,g)$ is an almost paracontact metric structure it is sufficient to show that the conditions \eqref{2.2} and \eqref{2.5} are valid. By using \eqref{2.6} and \eqref{3.3} we get $\eta (\xi )=1$. Then from \eqref{2.1} it follows $\varphi ^2\xi =0$ and hence $\varphi ^3\xi =0$. The equalities \eqref{2.6} and \eqref{2.7} imply $\eta (\varphi \xi )=0$. Thus, $\varphi \xi =\varphi \xi -\eta (\varphi \xi )\xi =\varphi ^2(\varphi \xi )=\varphi ^3\xi =0$. Replacing $Y$ with $\varphi Y$ in \eqref{2.7} and using  \eqref{2.1}, \eqref{2.6} we obtain  \eqref{2.5}.
\end{proof}
\begin{thm}\label{Theorem 3.1} Let $(M,g)$ be a  Walker 3-manifold.
In case $\epsilon =1$ there exist almost paracontact metric structures on $(M,g)$, which are defined with respect to the local basis $\left\{\partial _x,\partial _y,\partial _z\right\}$ as follows: 
\par
\[
 \xi =(\xi_1,\xi _2,\xi _3), \qquad  \eta =\xi _3dx+\xi _2dy+(\xi _1+f\xi _3)dz, 
\]
\begin{equation}\label{3.4}
\varphi  =\pm
\left(\begin{matrix}-\xi _2 & \xi _1+f\xi _3 &  -f\xi _2 \\  \, \,  \xi _3  & 0  & -\xi _1   \\  0 & - \xi _3 & \, \, \, \xi _2  \end{matrix}
\right) ,
\end{equation}
where $\xi _1,  \xi _2, \xi _3$ are smooth functions on $M$ satisfying
\begin{equation}\label{3.5}
 \xi _2^2+f\xi _3^2+2\xi _1\xi _3=1.
\end{equation}
In case $\epsilon =-1$ there exist no almost paracontact metric structures on $(M,g)$.
\end{thm}
\begin{rem}\label{Remark 3.1}
We  note that from condition \eqref{3.5} it follows that at least one of the functions $\xi _2$ and $\xi _3$ is non-zero at any point of 
$(M,g)$.
\end{rem}
\begin{proof}
With respect to the local basis $\left\{\partial _x,\partial _y,\partial _z\right\}$ we consider: a vector field $\xi =(\xi_1,\xi _2,\xi _3)$ on
$M$ such that \eqref{3.3} holds; the 1-form $\eta $, which is dual to $\xi $; a $(1,1)$ tensor field
\begin{equation*}
\varphi  =
\left(\begin{matrix} a _1 & b_1 & c _1 \\  a _2  & b_2  & c _2   \\  a_3  &  b _3 & c _3  \end{matrix}
\right) , 
\end{equation*}
where $a_i, b_i, c_i \in C^\infty (M), (i=1,2,3)$, which satisfies \eqref{2.1} and \eqref{2.7}.  By using \eqref{3.1}, the equality \eqref{3.3} becomes
\begin{equation}\label{3.6}
\epsilon \xi _2^2+f\xi _3^2+2\xi _1\xi _3=1, \quad \epsilon =\pm 1
\end{equation}
and $\eta $ is given by
\begin{equation}\label{3.7}
\eta =\xi _3dx+\epsilon \xi _2dy+(\xi _1+f\xi _3)dz.
\end{equation}
Further, we consider the cases $\epsilon =1$ and $\epsilon =-1$ in \eqref{3.1} separately.
\par
{\bf I.} $\epsilon =1$. The condition \eqref{2.7} implies $\varphi ^Tg=-g\varphi $, from where we obtain
\begin{equation}\label{3.8}
\varphi  =
\left(\begin{matrix} a _1 & b_1 & fa _1 \\  a _2  & 0  & -b _1+fa_2   \\  0  &  -a _2 & -a _1  \end{matrix}
\right) .
\end{equation}
Then we have
\begin{equation}\label{3.9}
\varphi ^2 =
\left(\begin{matrix} a _1^2+a_2b_1 & a_1(b_1-fa_2) & -b_1(b_1- fa _2) \\ \\  a_1a _2  & a_2(2b_1-fa_2)  & a_1b _1  \\  \\  -a _2^2 & a _1a_2 &
a_2(b_1-fa_2)+a_1^2 \end{matrix}
\right) .
\end{equation}
By using \eqref{2.1} and  \eqref{3.7}, we get
\begin{equation}\label{3.10}
\varphi ^2 =
\left(\begin{matrix} 1-\xi _1\xi _3 \, & -\xi _1\xi _2 & \, \, -\xi _1(\xi _1+f\xi _3) \\ \\   -\xi _2\xi _3 \,   & 1-\xi _2^2  &\, \,  -\xi _2(\xi _1+f\xi _3)  
\\  \\  
-\xi  _3^2 \, & -\xi _2\xi _3 & \, \, 1-\xi _3(\xi _1+f\xi _3) \end{matrix}
\right) .
\end{equation}
Equating the matrices in \eqref{3.9} and \eqref{3.10},  we obtain a system for $a_1$, $a_2$ and $b_1$, which is equivalent to the following system:
\begin{equation}\label{3.11}
\left|
\begin{array}{llllll}
a_2b_1=1-\xi _1\xi _3-a_1^2 \\
a_1a_2=-\xi _2\xi _3 \\
a_2^2=\xi _3^2 \\
a_1b_1=-\xi _2(\xi _1+f\xi _3)\\
2(1-\xi _1\xi _3-a_1^2)-f\xi _3^2=1-\xi _2^2 \\
b_1^2-f(1-\xi _1\xi _3-a_1^2)=\xi _1(\xi _1+f\xi _3)
\end{array}
\right. .
\end{equation}
Taking into account \eqref{3.6} by $\epsilon =1$, the system \eqref{3.11} is equivalent to
\begin{equation}\label{3.12}
\left|
\begin{array}{llllll}
a_1=\pm \xi _2 \\
a_2=\pm \xi _3 \\
a_1a_2=-\xi _2\xi _3 \\
a_2b_1=\xi _3(\xi _1+f\xi _3) \\
a_1b_1=-\xi _2(\xi _1+f\xi _3)\\
b_1=\pm (\xi _1+f\xi _3)
\end{array}
\right. .
\end{equation}
Since $a_1a_2=-\xi _2\xi _3$, for $a_1$ and $a_2$ the following cases are possible:
\par
{\bf Case 1.} $a_1=-\xi _2$, $a_2=\xi _3$. Then \eqref{3.12} becomes
\begin{equation}\label{3.13}
\left|
\begin{array}{lll}
\xi _3(b_1-\xi _1-f\xi _3)=0 \\
\xi _2(b_1-\xi _1-f\xi _3)=0 \\
b_1=\pm (\xi _1+f\xi _3)
\end{array}
\right. .
\end{equation}
\par
Case 1.1. If $b_1=\xi _1+f\xi _3$, then the first two equalities of \eqref{3.13} hold for any $\xi _2$ and $\xi _3$, satisfying \eqref{3.5}.
\par
\par
Case 1.2. If $b_1=-\xi _1-f\xi _3$, then  \eqref{3.13} becomes

\begin{equation*}
\left|
\begin{array}{lll}
\xi _3b_1=0 \\
\xi _2b_1=0
\end{array}
\right. .
\end{equation*}
Now, the assumption $b_1\neq 0$ implies $\xi _2=\xi _3=0$, which is a contradiction (see Remark \ref{Remark 3.1}). Hence, $b_1=0$.
\par
Summarizing the results in Case 1.1 and Case 1.2, we conclude that in Case 1 the solutions of \eqref{3.11} are $a_1=-\xi _2$, $a_2=\xi _3$,
$b_1=\xi _1+f\xi _3$ for any $\xi _2$ and $\xi _3$, satisfying \eqref{3.5}.
\par
{\bf Case 2.} $a_1=\xi _2$, $a_2=-\xi _3$. In this case, analogously as in Case 1, we we obtain that the solutions of \eqref{3.11} are $a_1=\xi _2$, $a_2=-\xi _3$, $b_1=-\xi _1-f\xi _3$ for any $\xi _2$ and $\xi _3$, satisfying \eqref{3.5}.
\par
Substituting   $a_1$, $a_2$, $b_1$ from Case 1 (respectively Case 2) in \eqref{3.8}, we receive  the matrix of $\varphi $ in \eqref{3.4} with a sign $(+)$
(respectively $(-)$) before it.
\par
The vector field $\xi $ and the 1-form $\eta $ are defined such that the equalities \eqref{3.3} and \eqref{2.6} are fulfilled. Moreover, the tensor field $\varphi $, given by \eqref{3.4}, satisfies \eqref{2.1} and \eqref{2.7}. Then from Lemma \ref{Lemma 3.1} it follows that the obtained structures $(\varphi ,\xi ,\eta ,g)$ in case $\epsilon =1$ are almost paracontact metric structures.
\par
{\bf II.} $\epsilon =-1$. Now, the equalities  \eqref{3.6}, \eqref{3.8}, \eqref{3.9}, \eqref{3.10}  take the form
\begin{equation}\label{3.14}
-\xi _2^2+f\xi _3^2+2\xi _1\xi _3=1 ,
\end{equation}
\begin{equation}\label{3.15}
\varphi  =
\left(\begin{matrix} a _1 & b_1 & fa _1 \\  a _2  & 0  & b _1+fa_2   \\  0  & a _2 & -a _1  \end{matrix}
\right) ,
\end{equation}
\begin{equation}\label{3.16}
\varphi ^2 =
\left(\begin{matrix} a _1^2+a_2b_1 & a_1(b_1+fa_2) & b_1(b_1+fa _2) \\ \\  a_1a _2  & a_2(2b_1+fa_2)  & -a_1b _1  \\  \\  a _2^2 & -a _1a_2 &  a_2(b_1+fa_2)+a_1^2 \end{matrix}
\right) ,
\end{equation}
\begin{equation}\label{3.17}
\varphi ^2 =
\left(\begin{matrix} 1-\xi _1\xi _3 \, & \xi _1\xi _2 & \, \, -\xi _1(\xi _1+f\xi _3) \\ \\   -\xi _2\xi _3 \,   & 1+\xi _2^2  &\, \,  -\xi _2(\xi _1+f\xi _3)  
\\  \\  
-\xi  _3^2 \, & \xi _2\xi _3 & \, \, 1-\xi _3(\xi _1+f\xi _3) \end{matrix}
\right) ,
\end{equation}
respectively. Equating the matrices in \eqref{3.16} and \eqref{3.17}, we obtain $-\xi _3^2=a_2^2$, which implies $a_2=\xi _3=0$. Substituting 
$\xi _3=0$ in \eqref{3.14}, we get $-\xi _2^2=1$, which is a contadiction. Hence, in case $\epsilon =-1$ there exist no almost paracontact metric structures on $(M,g)$.
\end{proof}
\section{Classes of almost paracontact metric  Walker 3- manifolds}\label{sec-4}
{\bf Hypothesis}: In sections 4 and 5, we assume $(M,\varphi ,\xi ,\eta ,g)$ to be an almost paracontact metric  Walker 3-manifold, 
endowed with one of the almost paracontact metric structures $(\varphi ,\xi ,\eta ,g)$, defined in Theorem \ref{Theorem 3.1} with respect to the local basis $\left\{\partial _x,\partial _y,\partial _z\right\}$  by  
\begin{equation}\label{3.18}
\begin{array}{lll}
\varphi  =
\left(\begin{matrix}-\xi _2 & \xi _1+f\xi _3 &  -f\xi _2 \\  \, \,  \xi _3  & 0  & -\xi _1   \\  0 & - \xi _3 & \, \, \, \xi _2  \end{matrix}
\right). 
\end{array}
\end{equation}
\begin{thm}\label{Theorem 4.1} Almost paracontact metric Walker 3-manifolds are never para-Sasakian (equivalently, K-paracontact metric).
\end{thm}
\begin{proof}
Let us suppose that an almost paracontact metric Walker 3-manifold $(M,\varphi ,\xi ,\eta ,g)$ is para-Sasakian.  
From Theorem \ref{Theorem 2.1} and \eqref{2.17} it follows that $\nabla _X\xi =\varphi X$, $X\in \chi(M)$. Hence,  $\nabla _{\partial _x}\xi =\varphi \partial _x$ and  $\nabla _{\partial _y}\xi =\varphi \partial _y$.  Using  
\eqref{3.2}, we have
\begin{equation*}
\left|
\begin{array}{lll}
(\xi _1)_x+\frac{1}{2}\xi _3f_x=-\xi _2 \\ \\
(\xi _2)_x=\xi _3 \\ \\
(\xi _3)_x =0
\end{array}
\right.  \quad \text{and} \quad
\left|
\begin{array}{lll}
(\xi _1)_y+\frac{1}{2}\xi _3f_y=\xi _1+f\xi _3 \\ \\
(\xi _2)_y=0 \\ \\
(\xi _3)_y =-\xi _3
\end{array}
\right. ,
\end{equation*}
respectively. From $(\xi _2)_y=0$ it follows that $\xi _2=\xi _2(x,z)$, which implies $(\xi _2)_x=(\xi _2)_x(x,z)$. Now, by using $(\xi _2)_x=\xi _3$,
we get $\xi _3=\xi _3(x,z)$. Since $(\xi _3)_x =0$, we obtain $\xi _3=\xi _3(z)$. Then $(\xi _3)_y =0$. On the other hand, we have
$(\xi _3)_y =-\xi _3$. Thus, $\xi _3=0$. 
Since $\nabla _{\partial _z}\xi =\varphi \partial _z$, we  get  $\left((\xi _1)_z+\frac{1}{2}\xi _1f_x+\frac{1}{2}f_y\right)\partial _x=-f\partial _x-\xi _1\partial _y+\partial _z$ in case $\xi _2=1$ and \\
$\left((\xi _1)_z+\frac{1}{2}\xi _1f_x-\frac{1}{2}f_y\right)\partial _x=f\partial _x-\xi _1\partial _y-\partial _z$ in case $\xi _2=-1$. Due to the linear independence of $\partial _x, \partial _y, \partial _z$, we obtained a contradiction in both cases for $ \xi _2$, completing the proof. 
\end{proof}
\begin{prop}\label{Proposition 4.1} An almost paracontact metric  Walker 3-manifold $(M,\varphi ,\xi ,\eta ,g)$ is paracontact metric if and only if the following conditions hold:
\begin{equation}\label{4.1}
\left|
\begin{array}{lll}
(\xi _2)_x-(\xi _3)_y=2\xi _3 \\ \\
(\xi _1)_x+\xi _3f_x+f(\xi _3)_x-(\xi _3)_z=-2\xi _2 \\ \\
(\xi _1)_y+\xi _3f_y+f(\xi _3)_y-(\xi _2)_z=2\xi _1
\end{array}
\right. .
\end{equation}
\end{prop}
\begin{proof}
The equalities \eqref{4.1} are obtained equating the non-zero components of $d\eta (X,Y)=\frac{1}{2}\left((\nabla _X\eta)Y-(\nabla _Y\eta)X\right)$ and $\phi (X,Y)=g(\varphi X,Y)$, which are given below
\begin{equation*}
\begin{array}{lll}
d\eta (\partial _x, \partial _y)=\frac{1}{2}((\xi _2)_x-(\xi _3)_y), \, \,
d\eta (\partial _x, \partial _z)=\frac{1}{2}((\xi _1)_x+\xi _3f_x+f(\xi _3)_x-(\xi _3)_z), \\ \\
d\eta (\partial _y, \partial _z)=\frac{1}{2}((\xi _1)_y+\xi _3f_y+f(\xi _3)_y-(\xi _2)_z), \\ \\ 
\phi (\partial _x, \partial _y)=\xi _3, \quad \phi (\partial _x, \partial _z)=-\xi _2, \quad \phi (\partial _y, \partial _z)=\xi _1.
\end{array}
\end{equation*}
\end{proof}
\begin{cor}\label{Corollary 4.1} Almost paracontact metric  Walker 3-manifolds for which either $\xi _3=0$ or $\xi _1=\xi _2=0$ are never paracontact metric.
\end{cor}
\begin{proof}
Let us suppose that an almost paracontact metric  Walker 3-manifold for which $\xi _3=0$ is paracontact metric. Then $\xi _2=\pm 1$ and the equalities \eqref{4.1} hold. Substituting $\xi _2$ and $\xi _3$ with $\pm 1$ and $0$ in \eqref{4.1}, respectively, we get 
$(\xi _1)_x=\pm 2$ and $(\xi _1)_y=2\xi _1$. Hence, $(\xi _1)_{xy}=0$ and $(\xi _1)_{yx}=\pm 4$, which is a contradiction. So, in case 
$\xi _3=0$ the statement in the corollary is proved.
In case $\xi _1=\xi _2=0$, using the first and third equalities of \eqref{4.1}, we obtain $f_y=2f$. On the other hand, from $\xi _3^2=\frac{1}{f}$, we have $2\xi _3(\xi _3)_y=-\frac{f_y}{f^2}$. Taking into account that $(\xi _3)_y=-2\xi _3$, we get $f_y=4f$ and hence $f=0$, which complete the proof.
\end{proof}
\begin{rem}\label{Remark 4.1}
Since in case $\xi _3=0$ we have $\eta =\pm dy+\xi _1dz$ and $d\eta =(\xi _1)_xdx\wedge dz+(\xi _1)_ydy\wedge dz$, the functions
$\xi _1$, $(\xi _1)_x$ and $(\xi _1)_y$ should be of class $C^\infty $.
If we do not take into account the last conditions, we get that $M$ is paracontact metric, which is not correct. Therefore, we regard  that the almost paracontact metric structure on a Walker 3-manifold, considered in \cite[Theorem 3.1., p. 393]{GC2}, is not paracontact.
\end{rem}
\begin{thm}\label{Theorem 4.2} An almost paracontact metric  Walker 3-manifold $(M,\varphi ,\xi ,\eta ,g)$ for which $\xi _2=0$ is paracontact metric if and only if
\begin{equation*}
\xi _3=\displaystyle{e^{-2y+\psi (z)}} \quad \text{and} \quad f=2\psi ^\prime (z)x+m(z),
\end{equation*}
where $x, y, z$ are local coordinates of $M$ and $\psi , m \in C^\infty (M)$.
\end{thm}
\begin{proof}
Let $M$ be paracontact metric. Then the equalities \eqref{4.1} hold. By $\xi _2=0$ they take the form
\begin{equation}\label{4.2}
(\xi _3)_y=-2\xi _3, 
\end{equation}
\begin{equation}\label{4.3}
(\xi _1)_x+\xi _3f_x+f(\xi _3)_x-(\xi _3)_z=0,
\end{equation}
\begin{equation}\label{4.4}
(\xi _1)_y+\xi _3f_y+f(\xi _3)_y=2\xi _1.
\end{equation}
Since $\xi _3\neq 0$, from \eqref{3.18} we get $\xi _1=\displaystyle\frac{1-f\xi _3^2}{2\xi _3}$. Then we find $(\xi _1)_x$ and $(\xi _1)_y$, which we put in \eqref{4.3} and \eqref{4.4}, respectively. So, \eqref{4.3}, \eqref{4.4} are equivalent to
\begin{equation}\label{4.5}
\xi _3^3f_x+(\xi _3)_x(f\xi _3^2-1)=2\xi _3^2(\xi _3)_z,
\end{equation}
\begin{equation}\label{4.6}
\xi _3^3f_y+(\xi _3)_y(f\xi _3^2-1)=2\xi _3(1-f\xi _3^2),
\end{equation}
respectively. By using \eqref{4.2}, from \eqref{4.6} we derive $f_y=0$. Next, we differentiate \eqref{4.5} with respect to $y$. Taking into account that $f_y=f_{yx}=f_{xy}=0$, $(\xi _3)_{xy}=(\xi _3)_{yx}=-2(\xi _3)_{x}$ and $(\xi _3)_{zy}=(\xi _3)_{yz}=-2(\xi _3)_{z}$, we obtain
\[
-3\xi _3^3f_x+(\xi _3)_x(1-3f\xi _3^2)+6\xi _3^2(\xi _3)_z=0.
\]
In the last equality we replace $2\xi _3^2(\xi _3)_z$ with the left side of \eqref{4.5} and we get $(\xi _3)_x=0$. Integrating \eqref{4.2} and 
because of $(\xi _3)_x=0$, we establish that $\xi _3=\displaystyle{e^{-2y+\psi (z)}}$, where $\psi $ is a function on $M$. Hence, 
$(\xi _3)_z=\xi _3\psi ^\prime (z)$. On the other hand, from \eqref{4.5} we derive $(\xi _3)_z=\frac{1}{2}\xi _3f_x$. So, we get $f_x=2\psi ^\prime (z)$.  Taking into account that $f_y=0$, we obtain $f=2\psi ^\prime (z)x+m(z)$, where $m$ is a function on $M$.
\par
Conversely, by direct calculations we check that the conditions \eqref{4.2}, \eqref{4.3} and \eqref{4.4} hold, which according to Proposition \ref{Proposition 4.1}  guarantees that $M$ is paracontact metric.
\end{proof}
In the remaining part of this section, we investigate the existence of other classes of almost paracontact metric  Walker 3-manifolds besides para-Sasakian and paracontact metric. For this purpose, we calculate the components $F_{xyz}=F(\partial _x,\partial _y,\partial _z)$ of the structure tensor $F$ by using \eqref{3.2}. Then for arbitrary vector fields  $X=x^1\partial _x+x^2\partial _y+x^3\partial _z$,  $Y=y^1\partial _x+y^2\partial _y+y^3\partial _z$, $Z=z^1\partial _x+z^2\partial _y+z^3\partial _z$, we obtain
\begin{equation}\label{4.7}
\begin{array}{lllll}
F(X,Y,Z)=(\xi _3)_x\left\{x^1y^1z^2-x^1y^2z^1\right\}-(\xi _2)_x\left\{x^1y^1z^3-x^1y^3z^1\right\}\\
+\left\{(\xi _1)_x+\displaystyle\frac{\xi _3f_x}{2}\right\}\left\{x^1y^2z^3-x^1y^3z^2\right\}+(\xi _3)_y\left\{x^2y^1z^2-x^2y^2z^1\right\}\\
-(\xi _2)_y\left\{x^2y^1z^3-x^2y^3z^1\right\}+\left\{(\xi _1)_y+\displaystyle\frac{\xi _3f_y}{2}\right\}\left\{x^2y^2z^3-x^2y^3z^2\right\} \\
+\left\{(\xi _3)_z-\displaystyle\frac{\xi _3f_x}{2}\right\}\left\{x^3y^1z^2-x^3y^2z^1\right\} \\
+\left\{-(\xi _2)_z+\displaystyle\frac{\xi _3f_y}{2}\right\}\left\{x^3y^1z^3-x^3y^3z^1\right\} \\
+\left\{(\xi _1)_z+\displaystyle{\frac{\xi _1f_x}{2}+\frac{\xi _2f_y}{2}+\frac{\xi _3f_z}{2}+\frac{\xi _3ff_x}{2}}\right\}\left\{x^3y^2z^3-x^3y^3z^2\right\}.
\end{array}
\end{equation}
By using \eqref{2.8} and the non-zero components $g^{11}=-f$, $g^{13}=g^{31}=1$, $g^{22}=1$ of the inverse matrix of $g$, we get
\begin{equation}\label{4.8}
\begin{array}{lllll}
\theta _F(\xi )=\xi _1\left\{(\xi _2)_x-(\xi _3)_y\right\}-\xi _2\left\{f(\xi _3)_x+(\xi _1)_x+\xi _3f_x-(\xi _3)_z\right\} \\ \\
\qquad +\xi _3\left\{f(\xi _2)_x+(\xi _1)_y+\xi _3f_y-(\xi _2)_z\right\}, \\ \\
\theta ^*_F(\xi )= \\
\xi _1\left\{(f\xi _3+\xi _1)(\xi _3)_x+\xi _2(\xi _2)_x-\xi _3(\xi _2)_y-\xi _3\left((\xi _3)_z-\displaystyle\frac{\xi _3f_x}{2}\right)\right\}\\
+\xi _2\left\{(f\xi _3+\xi _1)(\xi _3)_y-\xi _2(\xi _1)_x-\xi _2(\xi _3)_z+\xi _3\left((\xi _1)_y+\displaystyle\frac{\xi _3f_y}{2}\right)\right\}\\
+\xi _3\left\{-(f\xi _3+\xi _1)((\xi _1)_x+(\xi _2)_y)
+\xi _2(\xi _2)_z+\xi _3\left((\xi _1)_z+\displaystyle\frac{\xi _3f_z}{2}\right)\right\}.
\end{array}
\end{equation}
\begin{thm}\label{Theorem 4.3}
Let $M$ be an almost paracontact metric  Walker 3-manifold for which $\xi _1=\xi _2=0$ and $\xi _3=\displaystyle\frac{1}{\sqrt{f}}$, $f>0$. Then we have:
\par
(i) $M$ is never paracosymplectic, quasi-para-Sasakian, $\alpha $-paracosymplectic, $\alpha $-para-Kenmotsu, almost paracosymplectic and normal;
\par
(ii) $M$ is almost $\alpha $-paracosymplectic if and only if      
\begin{equation}\label{4.9}
\begin{array}{ll}
 f_y=0, \quad  f_z=-ff_x\neq 0
\end{array}
\end{equation}
and $M$ is never almost $\alpha $-para-Kenmotsu;
\par
(iii) $M\in \mathbb{G}_{12}$ if and only if $f_y=f_z=0$ and $f_x\neq 0$.
\end{thm}
\begin{proof}
We substitute $\xi _1=\xi _2=0$, $\xi _3=\displaystyle\frac{1}{\sqrt{f}}$ in \eqref{4.7} and obtain
\begin{equation}\label{4.10}
F(X,Y,\xi )=\displaystyle\frac{1}{2f}\left\{f_xx^1y^2+f_y\left(x^2y^2+x^3y^1\right)+(f_z+ff_x)x^3y^2\right\}.
\end{equation}
Replacing $X$ with $\xi $ in \eqref{4.10}, we derive
\begin{equation}\label{4.11}
F(\xi ,\xi ,Y)=\displaystyle-\frac{1}{2f\sqrt{f}}\left\{(f_z+ff_x)y^2+f_yy^1\right\}.
\end{equation}
From Proposition \ref{Proposition 2.1}, we have
\begin{equation}\label{4.12}
F(X,Y,Z)=(F^5+F^6+F^{10}+F^{12})(X,Y,Z).
\end{equation}
Then Lemma \ref{Lemma 2.2} and \eqref{4.12} imply $\theta _F(\xi )=\theta _{F^5}(\xi )$ and $\theta ^*_F(\xi )=\theta ^*_{F^6}(\xi)$.
Substituting $\xi _1=\xi _2=0$, $\xi _3=\displaystyle\frac{1}{\sqrt{f}}$ in \eqref{3.18} and \eqref{4.8}, we get
\begin{equation}\label{4.13}
\begin{array}{lllll}
\varphi =
\left(\begin {matrix} 0 & \sqrt{f} &  0 \cr \displaystyle\frac{1}{\sqrt{f}}  &  0  & 0   \cr 0 & -\displaystyle\frac{1}{\sqrt{f}} & 0  
\end {matrix}\right),  \qquad
\theta _F(\xi )=\displaystyle\frac{f_y}{f}, \qquad \theta ^*_F(\xi )=\displaystyle\frac{f_z}{2f\sqrt{f}}.
\end{array}
\end{equation}
By using \eqref{2.13}, \eqref{2.14},  \eqref{4.13} and $\eta (X)=\displaystyle-\frac{1}{\sqrt{f}}\left(x^1+fx^3\right)$, we obtain
\begin{equation}\label{4.14}
\begin{array}{ll}
F^5(X,Y,Z)=
\displaystyle\frac{f_y}{2f}\left\{\frac{1}{\sqrt{f}}\left(x^1y^3z^1+x^2y^2z^1-x^1y^1z^3-x^2y^1z^2\right)\right.\\ \\
\qquad \qquad \quad \, \left.+\sqrt{f}\left(x^2y^2z^3-x^2y^3z^2\right)\right\},
\end{array}
\end{equation}
\begin{equation}\label{4.15}
\begin{array}{ll}
F^6(X,Y,Z)=
\displaystyle-\frac{f_z}{4f\sqrt{f}}\left\{\frac{1}{f}\left(x^1y^2z^1-x^1y^1z^2\right)\right. \\ \\
\qquad \qquad \quad \, \left.+x^1y^2z^3+x^2y^3z^1-x^1y^3z^2-x^2y^1z^3\right\}.
\end{array}
\end{equation}
Further, by using \eqref{4.14}, \eqref{4.15} and \eqref{2.16}, \eqref{4.11}, we find $F^5(X,Y,\xi )$, $F^6(X,Y,\xi )$ and $F^{12}(X,Y,\xi )$, respectively. Then, taking into account \eqref{4.10} and \eqref{4.12}, we get
\begin{equation*}
\begin{array}{lllll}
F^{10}(X,Y,\xi )=F(X,Y,\xi )-(F^5+F^6+F^{12})(X,Y,\xi )\\
\qquad \qquad \quad \, \, \, =-\displaystyle\frac{f_z}{4f^2}\left\{x^1y^2+x^2y^1\right\}.
\end{array}
\end{equation*}
Now, the last equality and \eqref{2.15} imply
\begin{equation}\label{4.16}
\begin{array}{ll}
F^{10}(X,Y,Z)=
\displaystyle\frac{f_z}{4f^2}\left\{\frac{1}{\sqrt{f}}\left(x^1y^1z^2-x^1y^2z^1\right) \right.\\ \\
\qquad \qquad \quad \, \, \,  \left.+\sqrt{f}\left(x^1y^3z^2+x^2y^3z^1-x^1y^2z^3-x^2y^1z^3\right)\right\}.
\end{array}
\end{equation}
By using \eqref{2.16} and \eqref{4.11}, we obtain 
\begin{equation}\label{4.17}
\begin{array}{lllll}
F^{12}(X,Y,Z)=
\displaystyle\frac{1}{2f\sqrt{f}}\left\{f_y\left[x^1y^1z^3-x^1y^3z^1\right.\right.\\ \\
\left.\left.+f\left(x^3y^1z^3-x^3y^3z^1\right)\right]\right.\\ \\
+\displaystyle\frac{f_z+ff_x}{f}\left[x^1y^2z^1-x^1y^1z^2+f\left(x^1y^2z^3-x^1y^3z^2-x^3y^1z^2\right.\right.\\ \\
\left.\left.\left.+x^3y^2z^1\right)+f^2\left(x^3y^2z^3-x^3y^3z^2\right)\right]
\right\}.
\end{array}
\end{equation}
\par
(i) Let us suppose that $M$ belongs to some of the classes $\mathbb{G}_0$, $\mathbb{G}_5$, $\mathbb{G}_6$, $\mathbb{G}_{10}$ and 
$\mathbb{G}_5\oplus \mathbb{G}_6$. Then $F=F^i$ ($i=0,5,6,10$) and $F=F^5+F^6$, respectively. By using \eqref{4.12},  \eqref{4.14},
\eqref{4.15}, \eqref{4.16} and \eqref{4.17},
we obtain $f_x=f_y=f_z=0$. From the last conditions it follows that $f=const$, which is a contradiction. Hence, $M$ never belongs to the classes $\mathbb{G}_0$, $\mathbb{G}_5$, $\mathbb{G}_6$, $\mathbb{G}_{10}$ and 
$\mathbb{G}_5\oplus \mathbb{G}_6$. By using Theorem \ref{Theorem 2.1} and Proposition \ref{Proposition 2.2},  we complete the proof.
\par
(ii) Analogously as in (i), one can easy see that $M\in \mathbb{G}_6\oplus \mathbb{G}_{10}$ if and only if  the conditions \eqref{4.9} hold. 
Taking into account that  $M$ never belongs to $\mathbb{G}_{10}$ and  Proposition \ref{Proposition 2.2}, we obtain that $M$ is almost 
$\alpha $-paracosymplectic or almost $\alpha $-para-Kenmotsu if and only if \eqref{4.9} are fulfilled. Now, we will show that from  \eqref{4.9} it follows that $\theta ^*_F(\xi )$ is a function on $M$. If  we suppose that $\theta ^*_F(\xi )$ is a constant, then 
$\displaystyle\frac{df}{2f\sqrt{f}}=Cdz$, $C\neq 0$. By integrating the last equality and taking into account that $f_y=0$, we get
$\sqrt{f}=-\displaystyle\frac{1}{Cz+\psi (x)}$. Hence, the equality $f_z=-ff_x$ becomes $C=-\displaystyle\frac{\psi ^\prime (x)}{(Cz+\psi (x))^2}$, which leads to a contradiction. Thus, the proof is completed.
\par
(iii) We establish the equivalence in (iii) by using  \eqref{4.12},  \eqref{4.14}, \eqref{4.15}, \eqref{4.16} and \eqref{4.17}.
\end{proof}
\begin{thm}\label{Theorem 4.4}
Let $M$ be an almost paracontact metric  Walker 3-manifold for which $\xi _3=0$ and $\xi _2=1$. Then we have:
\par
(i) $M$ does not belong to the following classes: $\mathbb{G}_5$, $\mathbb{G}_5\oplus \mathbb{G}_{10}$, $\mathbb{G}_5\oplus \mathbb{G}_{12}$, $\mathbb{G}_5\oplus \mathbb{G}_{10}\oplus \mathbb{G}_{12}$, $\mathbb{G}_6$, $\mathbb{G}_6\oplus \mathbb{G}_{10}$, $\mathbb{G}_6\oplus \mathbb{G}_{12}$, $\mathbb{G}_6\oplus \mathbb{G}_{10}\oplus \mathbb{G}_{12}$;
\par
(ii) $M$ is paracosymplectic if and only if 
\[
(\xi _1)_x=(\xi _1)_y=0 \quad \text{and} \quad 2(\xi _1)_z+\xi _1f_x+f_y=0;
\]
\par
(iii) $M$ is normal if and only if 
\[
(\xi _1)_x\neq 0, \, \, (\xi _1)_y=-\xi _1(\xi _1)_x \quad \text{and} \quad 2(\xi _1)_z+\xi _1f_x+f_y+(\xi _1)_x\left(\xi _1^2-f\right)=0;
\]
\par
(iv) $M$ is almost paracosymplectic if and only if $(\xi _1)_x=(\xi _1)_y=0$ and
$2(\xi _1)_z+\xi _1f_x+f_y\neq 0$;
\par
(v) $M\in \mathbb{G}_{12}$ if and only if 
\[
(\xi _1)_x=0, \, \, \,   (\xi _1)_y\neq 0 \quad \text{and} \quad 2\xi _1(\xi _1)_y-2(\xi _1)_z-\xi _1f_x-f_y=0.
\]
\end{thm}
\begin{proof}
In the same way as in Theorem \ref{Theorem 4.3}, by $\xi _2=1$ and $\xi _3=0$, we get
\begin{equation*}
\begin{array}{llllll}
F(X,Y,\xi )=-(\xi _1)_xx^1y^3-(\xi _1)_yx^2y^3-\left\{(\xi _1)_z+\displaystyle{\frac{\xi _1f_x}{2}+\frac{f_y}{2}}\right\}x^3y^3,  \\ \\
F(\xi ,\xi ,Y)=\left\{\xi _1(\xi _1)_x+(\xi _1)_y\right\}y^3, \qquad
\theta _F(\xi )=\theta ^*_F(\xi )=-(\xi _1)_x, \\ \\
\varphi =
\left(\begin {matrix} -1 & \xi _1 &  -f \cr \, \,  \, 0 &  0  & -\xi _1   \cr  \, \,\,  0 & 0 & \, \, 1  
\end {matrix}\right), \qquad \eta (X)=x^2+\xi _1x^3.
\end{array}
\end{equation*}
Then we obtain
\begin{equation*}
\begin{array}{lllllll}
F^5(X,Y,Z)=\\
\displaystyle-\frac{(\xi _1)_x}{2}\left\{x^1y^3z^2+x^3y^1z^2-x^1y^2z^3-x^3y^2z^1-f\left(x^3y^2z^3-x^3y^3z^2\right)\right.
\\ \\
\left.+\xi _1\left(x^2y^2z^3+x^3y^1z^3-x^2y^3z^2-x^3y^3z^1\right)
\right\},\\ \\
F^6(X,Y,Z)=\\
\displaystyle\frac{(\xi _1)_x}{2}\left\{x^1y^2z^3+x^3y^1z^2-x^1y^3z^2-x^3y^2z^1\right.
\\ \\
\left.+\xi _1\left[x^2y^3z^2+x^3y^1z^3-x^2y^2z^3-x^3y^3z^1+\xi _1\left(x^3y^3z^2-x^3y^2z^3\right)\right]
\right\},\\ \\
F^{10}(X,Y,Z)=\\
\left\{\xi _1(\xi _1)_y+\displaystyle{\frac{(\xi _1)_x(\xi _1^2+f)}{2}-(\xi _1)_z-\frac{\xi _1f_x}{2}-\frac{f_y}{2}}\right\}\left\{x^3y^3z^2-x^3y^2z^3\right\},\\ \\
F^{12}(X,Y,Z)=\\
\left\{\xi _1(\xi _1)_x+(\xi _1)_y\right\}\left\{x^2y^2z^3-x^2y^3z^2+\xi _1\left(x^3y^2z^3-x^3y^3z^2\right)\right\}.\\ \\
\end{array}
\end{equation*}
\par
(i) Let us suppose that $M$ belongs to some of the classes $\mathbb{G}_5$, $\mathbb{G}_5\oplus \mathbb{G}_{10}$, $\mathbb{G}_5\oplus \mathbb{G}_{12}$, $\mathbb{G}_5\oplus \mathbb{G}_{10}\oplus \mathbb{G}_{12}$, (resp. $\mathbb{G}_6$,  $\mathbb{G}_6\oplus \mathbb{G}_{10}$, $\mathbb{G}_6\oplus \mathbb{G}_{12}$, $\mathbb{G}_6\oplus \mathbb{G}_{10}\oplus \mathbb{G}_{12}$). Then we have $\theta _F(\xi )= \theta _{F^5}(\xi )\neq 0$ and $\theta ^*_F(\xi )=0$ (resp. $\theta _F(\xi )=0$ and 
$\theta ^*_F(\xi )=\theta ^*_{F^6}(\xi )\neq 0$). Since $\theta _F(\xi )=\theta ^*_F(\xi )$, we obtained a contradiction, completing the proof.
\par
One can easily prove the assertions (ii), (iii), (iv) and (v) by using \eqref{4.12} and the expressions for $F^5$, $F^6$, $F^{10}$ and $F^{12}$.
\end{proof}
Analogously to Theorem \ref{Theorem 4.4} we prove the following 
\begin{thm}\label{Theorem 4.5}
Let $M$ be an almost paracontact metric Walker 3-manifold for which $\xi _3=0$ and $\xi _2=-1$. Then we have:
\par
(i) $M$ does not belong to the following classes: $\mathbb{G}_5$, $\mathbb{G}_5\oplus \mathbb{G}_{10}$, $\mathbb{G}_5\oplus \mathbb{G}_{12}$, $\mathbb{G}_5\oplus \mathbb{G}_{10}\oplus \mathbb{G}_{12}$, $\mathbb{G}_6$, $\mathbb{G}_6\oplus \mathbb{G}_{10}$, $\mathbb{G}_6\oplus \mathbb{G}_{12}$, $\mathbb{G}_6\oplus \mathbb{G}_{10}\oplus \mathbb{G}_{12}$;
\par
(ii) $M$ is paracosymplectic if and only if 
\[
(\xi _1)_x=(\xi _1)_y=0 \quad \text{and} \quad 2(\xi _1)_z+\xi _1f_x-f_y=0;
\]
\par
(iii) $M$ is normal if and only if 
\[
(\xi _1)_x\neq 0, \, \, (\xi _1)_y=\xi _1(\xi _1)_x \quad \text{and} \quad 2(\xi _1)_z+\xi _1f_x-f_y+(\xi _1)_x\left(\xi _1^2-f\right)=0;
\]
\par
(iv) $M$ is almost paracosymplectic if and only if $(\xi _1)_x=(\xi _1)_y=0$ and
$2(\xi _1)_z+\xi _1f_x-f_y\neq 0$;
\par
(v) $M\in \mathbb{G}_{12}$ if and only if 
\[
(\xi _1)_x=0, \, \, \,   (\xi _1)_y\neq 0 \quad \text{and} \quad 2\xi _1(\xi _1)_y+2(\xi _1)_z+\xi _1f_x-f_y=0.
\]
\end{thm}
As an immediate consequence of assertion (i) in Theorem \ref{Theorem 4.4} and Theorem \ref{Theorem 4.5} is the following
\begin{cor}\label{Corollary 4.2} An almost paracontact metric  Walker 3-manifold for which $\xi _3=0$ and $\xi _2=\pm 1$ is never quasi-para-Sasakian, almost $\alpha $-paracosymplectic, $\alpha $-paracosymplectic, almost $\alpha $-para-Kenmotsu and $\alpha $-para-Kenmotsu.
\end{cor}
\begin{rem}\label{Remark 4.2}
From Theorem \ref{Theorem 4.4} and Theorem \ref{Theorem 4.5} we have that an almost paracontact metric  Walker 3-manifold $M^\prime $ for which $\xi _3=0$ does not belong to the classes $\mathbb{G}_5$ and $\mathbb{G}_5\oplus \mathbb{G}_{10}$. 
Since $\overline {\mathbb{G}}_5$ is a subclass of $\mathbb{G}_5$ and $\overline {\mathbb{G}}_5\oplus \mathbb{G}_{10}$ is a subclass of 
$\mathbb{G}_5\oplus \mathbb{G}_{10}$, it follows that $M^\prime \notin \overline {\mathbb{G}}_5$ and $M^\prime \notin \overline {\mathbb{G}}_5\oplus \mathbb{G}_{10}$. By using \eqref{4.12} $\div $ \eqref{4.17}, it is easy to see that an almost paracontact metric Walker 3-manifold $M''$ for which $\xi _1=\xi _2=0$ also does not belong to $\mathbb{G}_5$ and $\mathbb{G}_5\oplus \mathbb{G}_{10}$. Taking into account  assertions (b) and (c) in Theorem \ref{Theorem 2.1}, we conclude that $M^\prime $ and $M''$ are never paracontact metric and para-Sasakian, which agrees with the results obtained in Corollary \ref{Corollary 4.1} and Theorem \ref{Theorem 4.3}.
\end{rem}
Using Theorem \ref{Theorem 4.2}, many examples of paracontact metric  Walker 3-manifolds can be obtained. Further, applying Theorem \ref{Theorem 4.3} and Theorem \ref{Theorem 4.4}, we construct  examples of paracosymplectic, normal, almost $\alpha $-paracosymplectic, almost paracosymplectic  Walker 3-manifolds and $\mathbb{G}_{12}$-manifolds. 
\begin{example}\label{Example 4.1} 
According to Theorem \ref{Theorem 3.1}, an almost paracontact metric structure $(\varphi ,\xi ,\eta ,g)$ on a  Walker 3-manifold $M$ is determined with respect to the local basis $\left\{\partial _x,\partial _y,\partial _z\right\}$ only by a unit space-like vector field $\xi =\xi_1\partial _x+\xi _2\partial _y+\xi _3\partial _z$, globally defined on $M$ and a function $f(x,y,z)$ on $M$. Therefore, in the following examples we give only the coordinates $\xi _1$,  $\xi _2$, $\xi _3$ of $\xi $ and the function $f$ on $M$. By using \eqref{4.7} and \eqref{4.8}, we find the corresponding tensor $F$ and the 1-forms $\theta _F(\xi )$, $\theta ^*_F(\xi )$, respectively.
By virtue of  \eqref{2.13} $\div $ \eqref{2.16}, we determine the class to which $M$ belongs.
\par
(i) $\xi _1=\displaystyle{e^{\frac{C-\int {\psi (z)dz}}{2}}}$, $\xi _2=1$, $\xi _3=0$, $f=\psi (z)x+m(z)$, where $C=const$, $\psi , m \in C^\infty (M)$.
Then $F(X,Y,Z)=0$ and hence, $M\in \mathbb{G}_0$ or equivalently $M$ is paracosymplectic.
\par
(ii) $\xi _1=\displaystyle\frac{x}{y}$, $\xi _2=1$, $\xi _3=0$, $f=\displaystyle\frac{x^2}{y^2}$, $y\neq 0$. Then $\theta _F(\xi )=\theta ^*_F(\xi )=-\displaystyle\frac{1}y{}$,
\begin{equation*}
\begin{array}{lll}
F(X,Y,Z)=\displaystyle\frac{1}{y}\left\{x^1y^2z^3-x^1y^3z^2+\frac{x}{y}\left(x^2y^3z^2-x^2y^2z^3\right)\right\}.
\end{array}
\end{equation*}
By using \eqref{2.13} and \eqref{2.14}, we get $M\in \mathbb{G}_5\oplus \mathbb{G}_6$ or equivalently $M$ is normal.
\par
(iii) $\xi _1=e^{\frac{Cz+C_1}{2}}$, $\xi _2=1$, $\xi _3=0$, $f=Cx+\psi (z)$, where $C$, $C_1$ are constants,  $C\neq 0$ and $\psi  \in C^\infty (M)$. Then 
$\theta _F(\xi )=\theta ^*_F(\xi )=0$, 
\[
F(X,Y,Z)=-Ce^{\frac{Cz+C_1}{2}}\left\{x^3y^3z^2-x^3y^2z^3\right\}.
\]
By using \eqref{2.15}  we get $M\in \mathbb{G}_{10}$ or equivalently $M$ is almost paracosymplectic.
\par
(iv) $\xi _1=\xi _2=0$, $\xi _3=\displaystyle\frac{1}{\sqrt{\frac{x}{z}}}$, $f=\displaystyle\frac{x}{z}$, where $x<0, z<0$ $\cup $ $x>0, z>0$. 
Then $\theta _F(\xi )=0$, $\theta ^*_F(\xi )=-\displaystyle\frac{1}{2z\sqrt{\frac{x}{z}}}$,
\[
F(X,Y,Z)=-\displaystyle\frac{1}{2z\sqrt{\frac{x}{z}}}\left\{\frac{z}{x}\left(x^1y^1z^2-x^1y^2z^1\right)+x^1y^3z^2-x^1y^2z^3\right\}.
\]
By using \eqref{2.14} and \eqref{2.15}, we get $M\in \mathbb{G}_6\oplus \mathbb{G}_{10}$, but $M\notin \mathbb{G}_{10}$ or equivalently $M$ is almost $\alpha $-paracosymplectic.
\par
(v) $\xi _1=\frac{C}{2}y+C_1$, $\xi _2=1$, $\xi _3=0$, $f=Cx+\psi (z)$, where $C$, $C_1$ are constants, $C\neq 0$ and 
$\psi  \in C^\infty (M)$. Then 
$\theta _F(\xi )=\theta ^*_F(\xi )=0$,
\[
F(X,Y,Z)=
\frac{C}{2}\left\{x^2y^2z^3-x^2y^3z^2+\left(\frac{C}{2}y+C_1\right)\left(x^3y^2z^3-x^3y^3z^2\right)\right\}.
\]
By using \eqref{2.16}, we get $M\in \mathbb{G}_{12}$.
\end{example}
\section{Curvature properties of almost paracontact metric  Walker 3-manifolds}\label{sec-5}
Curvature properties of a Walker 3-dimensional manifold $(M,g)$ are well studied in \cite{CRV}.  Below we give the non-zero components of the curvature tensor $R$, the Ricci tensor $\rho $, the Ricci operator $Q$ and the scalar curvature $scal$ with respect to the local basis 
$\left\{\partial _x,\partial _y,\partial _z\right\}$, obtained in \cite{CRV}:
\begin{equation}\label{5.1}
\begin{array}{llll}
R(\partial _x,\partial _z)\partial _x=-\displaystyle{\frac{1}{2}f_{xx}\partial _x, \quad R(\partial _x,\partial _z)\partial _y=-\frac{1}{2}f_{xy}\partial _x,}\\ \\
R(\partial _x,\partial _z)\partial _z=-\displaystyle{\frac{1}{2}ff_{xx}\partial _x+\frac{1}{2}f_{xy}\partial _y+\frac{1}{2}f_{xx}\partial _z,}\\ \\

R(\partial _y,\partial _z)\partial _x=-\displaystyle{\frac{1}{2}f_{xy}\partial _x, \quad R(\partial _y,\partial _z)\partial _y=-\frac{1}{2}f_{yy}\partial _x,}\\ \\
R(\partial _y,\partial _z)\partial _z=-\displaystyle{\frac{1}{2}ff_{xy}\partial _x+\frac{1}{2}f_{yy}\partial _y+\frac{1}{2}f_{xy}\partial _z};
\end{array}
\end{equation}
\begin{equation}\label{5.2}
\begin{array}{lll}
\rho  =
\left(\begin{matrix} 0 & 0 & \frac{1}{2}f_{xx} \\ \\ 0  & 0  & \frac{1}{2}f_{xy}   \\  \\ \frac{1}{2}f_{xx}  & \frac{1}{2}f_{xy} & 
\frac{1}{2}(\epsilon ff_{xx}-f_{yy})  \end{matrix}
\right);
\end{array}
\end{equation}
\begin{equation}\label{5.3}
\begin{array}{lll}
Q=
\left(\begin {matrix} \frac{1}{2}f_{xx} & \frac{1}{2}f_{xy} &  -\frac{\epsilon }{2}f_{yy} \\ \\  0  &  0  & \frac{\epsilon }{2}f_{xy}  
 \\ \\ 0 & 0 &  \frac{1}{2}f_{xx} 
\end {matrix}\right),
\end{array}
\end{equation}
\begin{equation}\label{5.33}
scal={\rm trace} \, Q=f_{xx}.
\end{equation}
\begin{rem}\label{Remark 5.1}
We note that \eqref{5.1},  \eqref{5.2} and  \eqref{5.3} were obtained when $R$ is defined by the equality $R(X,Y)Z=\nabla _{[X,Y]}Z-\nabla _X\nabla _YZ+\nabla _Y\nabla _XZ$. In the remaining part of this paper, having in mind Theorem \ref{Theorem 3.1}, we take $\epsilon =1$ in  \eqref{5.2} and  \eqref{5.3}.
\end{rem}
It is well known that in Riemannian geometry, self-adjoint operators (called symmetric operators) have real eigenvalues and they are diagonalizable. However, in Lorentzian geometry, the self-adjoint operators are not always diagonalizable. They are categorized by
the Segre classification  according to their Jordan canonical form, taking into account the Lorentzian metric. 
\par
In a 3-dimensional Lorentzian manifold $(M,g)$ the Ricci operator $Q$, defined by $g(QX,Y)=\rho (X,Y)$, is self-adjoint. For $Q$ the following four different cases are possible, known as Segre types (see \cite{BN, BD, GC2}):
\par
1. {\it Segre type} $\{11; 1\}$: the Ricci operator itself is symmetric and so, diagonalizable. The comma is used to separate
the spacelike and timelike degenerate Jordan blocks, meaning that there are two distinct Jordan blocks of 1-dimension, with the same eigenvalue in the spacelike direction and then 1 in the timelike direction. In the degenerate case, at least two of the Ricci eigenvalues coincide.
\par
2. {\it Segre type} $\{1z\bar z\}$: the Ricci operator has one real and two complex conjugate eigenvalues.
\par
3. {\it Segre type} $\{21\}$: the Ricci operator has two real eigenvalues (coinciding in the degenerate case), one of which
has multiplicity two and each associated to a one-dimensional eigenspace.
\par
4. {\it Segre type} $\{3\}$: the Ricci operator has three equal eigenvalues, associated to a one-dimensional eigenspace.
\par
In case $(M,g)$ is a Walker 3-manifold, in  \cite{GC1} it is proved that the Ricci operator is never of Segre type 
$\{1z\bar z\}$ and necessary and sufficient conditions it to belong to the rest three Segre types are given. For later use, here we recall only that $(M,g)$ is of  Segre type  $\{11; 1\}$ if and only if at any point of $M$ either the conditions
\begin{equation}\label{5.4}
f^2_{xy}-f_{xx}f_{yy}=0\neq f_{xx}\quad \text{(degenerate case)},
\end{equation}
hold, or
\begin{equation*}
f_{xx}=f_{yy}=f_{zz}=0 \quad \text{(degenerate and flat case)}.
\end{equation*}
By using \eqref{5.3}, we find the Ricci eigenvalues 
\[
\lambda _1=0, \quad \lambda _2= \lambda _3=\frac{1}{2}f_{xx}.
\]
Now, if $(M,g)$ is a  Walker 3-manifold of  Segre type  $\{11; 1\}$ degenerate case, i.e. the conditions \eqref{5.4} hold, then the 
vector fields
\begin{equation}\label{5.5}
N=-\frac{f_{xy}}{f_{xx}}\partial _x+\partial _y, \quad \text{and} \quad V_1=\partial _x, \quad V_2=\frac{f_{xy}}{f_{xx}}\partial _y+\partial _z
\end{equation}
are eigenvectors of $Q$, corresponding to $\lambda _1$ and $\lambda _2=\lambda _3$, respectively.
\par
Since a  Walker 3-manifold is Einstein if and only if it is flat (see \cite{CRV}), we focus to investigate $\eta $-Einstein almost paracontact metric  Walker 3-manifolds. For this purpose, we recall the following notion, introduced by Okumura in the Sasakian context, which was then extended to almost contact and paracontact  metric manifolds:
\begin{defn}\label{Definition 5.0} 
An almost paracontact metric manifold $M$ is said to be $\eta $-{\it Einstein} if there exist two smooth functions $a$, $b$ on $M$, such that 
\begin{equation}\label{5.66}
\rho =ag+b\eta \otimes \eta .
\end{equation}
When moreover,  both the smooth functions $a$ and $b$ are non-zero, we call $M$ pure $\eta $-Einstein.
\end{defn}
In the above relation \eqref{5.66}, the two smooth functions $a$ and $b$ quantify how the curvature of the manifold is directionally distributed at each point, i.e.  they measure the geometric deviation from an ideal Einstein manifold, where the curvature spread in a uniform way in all directions. Precisely, the function $a$ controls the Ricci curvature within the contact distribution ${\mathbb D}$, i.e. $a(p)$ represents the isotropy within the transversal space, which means the average curvature of the manifold at the point $p$, looking perpendicular to the Reeb direction. On the other hand, the function $b$ represents the anisotropy induced by $\xi$, i.e.  $b(p)$ controls the strength of the correction applied by the $\xi $ axis at any point $p$. Therefore, one obtains $\rho (\xi ,\xi )=a+b$. Obviously, when $a$ is constant and $b=0$, the manifold becomes Einstein. Moreover, when both function $a$ and $b$ vanish identically, the manifold is Ricci flat.
\par
The reason for which we studied the case when $\xi_3=0$ in Section 4 was to prepare the context of the following:
\begin{thm}\label{Theorem 5.1}
An almost paracontact metric Walker 3-manifold 
$(M,\varphi ,\xi ,\eta ,g)$ is pure $\eta $-Einstein if and only if the following conditions hold
\begin{equation}\label{5.6}
\begin{array}{lc}
\xi _1=-\displaystyle\frac{f_{xy}}{f_{xx}}, \, \,  \xi _2=1, \,\,\xi _3=0  \quad  \text{or} \quad  \xi _1=\displaystyle\frac{f_{xy}}{f_{xx}}, \, \,   \xi _2=-1, \,\, \xi _3=0;\\ \\
f_{xy}^2-f_{xx}f_{yy}=0\neq f_{xx},
\end{array}
\end{equation}
or equivalently the Ricci operator $Q$ of $M$ is of Segre type $\{11,1\}$ degenerate case and  $\xi =\pm N$, where $N$ is the eigenvector of $Q$ for the eigenvalue $\lambda _1=0$, given in \eqref{5.5}.
\end{thm}
\begin{proof}
For the components of the tensor $ag+b\eta \otimes \eta $ with respect to  
$\left\{\partial _x,\partial _y,\partial _z\right\}$, we have
\begin{equation}\label{5.7}
\begin{array}{ll}
ag+b\eta \otimes \eta =\\ \\
\left(\begin{matrix} b\xi _3^2 &  b\xi _2 \xi _3 & a+ b\xi _3(\xi _1+f\xi _3) \\ \\ b\xi _2 \xi _3  & a+b\xi _2^2  & b\xi _2(\xi _1+f\xi _3)   \\  \\ 
a+ b\xi _3(\xi _1+f\xi _3)  & b\xi _2(\xi _1+f\xi _3) & fa+b(\xi _1+f\xi _3)^2  \end{matrix}
\right), 
\end{array}
\end{equation}
where $a$, $b$ are non-zero smooth functions on $M$. Equating the matrices in \eqref{5.2} and \eqref{5.7} and taking into account that $b\neq 0$, we get $\xi _3=0$. The last one implies $\xi _2=\pm 1$. Then, by standard calculations,  we obtain
\begin{equation}\label{5.8}
\begin{array}{l}
a=-b=\displaystyle\frac{1}{2}f_{xx}\neq 0
\end{array}
\end{equation}
and \eqref{5.6}, which  complete the proof. 
\end{proof}
In \cite{NM} a K\"ahler curvature tensor of an almost contact manifold with B-metric was defined. Here  we give an analogous definition for a K\"ahler curvature tensor of an almost paracontact metric manifold.
\begin{defn}\label{Definition 5.1} The curvature tensor $R$ of an almost paracontact metric manifold is said to be  K\"ahler tensor  if it satisfies the K\"ahler property 
\[
R(X,Y,Z,W)=-R(X,Y,\varphi Z,\varphi W).
\]
\end{defn}
It is easy to show that the above equality is equivalent to the condition 
\begin{equation}\label{5.9}
R(X,Y)\varphi Z=\varphi R(X,Y)Z.
\end{equation}
\begin{thm}\label{Theorem 5.2}
Let $(M,\varphi ,\xi ,\eta ,g)$ be an almost paracontact metric Walker 3-manifold. Then the following assertions are equivalent:
\par
(i) $Q\circ \varphi =\varphi \circ Q$;
\par
(ii)  $M$  is either flat, or $M$ is $\eta $-Einstein;
\par
(iii) the curvature tensor $R$ is a K\"ahler tensor, i.e. \eqref{5.9} holds;
\par
(iv) the Ricci tensor $\rho $ satisfies
\begin{equation}\label{5.10}
\rho (\varphi X,\varphi Y)=-\rho (X,Y);
\end{equation}
\par
(v) the curvature tensor $R$ satisfies
\begin{equation}\label{5.11}
R(X,Y)\xi =0.
\end{equation}
\end{thm}
\begin{proof}
(i) $\Leftrightarrow$ (ii): By straightforward calculations, using  \eqref{5.3}, we obtain that $Q\circ \varphi =\varphi \circ Q$ if and only if either $f_{xx}=f_{yy}=f_{xy}=0$, or the conditions \eqref{5.6} hold. Taking into account \eqref{5.1} and Theorem \ref{Theorem 5.1}, we complete the proof.
\par
(ii) $\Leftrightarrow$ (iii): By using  \eqref{5.1}, one  easily checks that \eqref{5.9} is fulfilled if and only if either $f_{xx}=f_{yy}=f_{xy}=0$, or the conditions \eqref{5.6} hold, which proves this equivalence.
\par
(iv) $\Rightarrow $ (i): Replacing $X$ with $\xi $ in \eqref{5.10}, we obtain $g(Q\xi ,Y)=0$. The last one implies $Q\xi  =0$. From $g(Q(\varphi X),\varphi Y)=-g(QX,Y)$ we get $\varphi (Q(\varphi X))=QX$. Hence, $\varphi ^2(Q(\varphi X))=\varphi (QX)$. By using $\eta (Q(\varphi X))=
g(\varphi X,Q\xi )=0$, we derive $Q(\varphi X)=\varphi (QX)$.
\par
We note  that the implication (iv) $\Rightarrow $ (i) is valid for an arbitrary almost paracontact metric structure $(\varphi ,\xi ,\eta ,g)$, but the converse is not true in general.
\par
(i) $\Rightarrow $ (iv): Now, let us asuume that $Q\circ \varphi =\varphi \circ Q$. Thus, (ii) holds. If $M$ is flat, then $Q\xi =0$. If $M$ is 
$\eta $-Einstein, then from Theorem \ref{Theorem 5.1} it follows that $\xi =\pm N$. Since $N$ is an eigenvector of $Q$ for the eigenvalue 
$\lambda _1=0$, we also have that $Q\xi =0$. Hence, in both cases $\eta (QX)=0$ for any $X\in \chi (M)$. So, we obtain
$\rho (\varphi X,\varphi Y)=g(Q(\varphi X),\varphi Y)=g(\varphi (Q X),\varphi Y)=-g(QX,Y)=-\rho (X,Y)$.
\par
(iii) $\Rightarrow $ (v): This  implication is true for an arbitrary almost paracontact metric structure $(\varphi ,\xi ,\eta ,g)$. Indeed, replacing
$Z$ with $\xi $ in \eqref{5.9}, we derive $\varphi ^2R(X,Y)\xi =0$. Since $\eta (R(X,Y)\xi )=0$, we obtain $R(X,Y)\xi =0$.
\par
(v) $\Rightarrow $ (ii):  By using \eqref{5.1} and Theorem \ref{Theorem 5.1}, we  check directly that the implication holds.
\end{proof}
We remark that on a paracosymplectic manifold (i.e. on a $\mathbb{G}_0$-manifold) we have $\nabla \varphi =0$. Hence, the curvature tensor of a paracosymplectic manifold is a K\"ahler tensor. So, as a  consequence of  Theorem \ref{Theorem 5.2}, we obtain
\begin{cor}\label{Corollary 5.1} A paracosymplectic  Walker 3-manifold is either flat, or $\eta $-Einstein.
\end{cor}
At the end of the section, we deal with the scalar curvature and sectional curvatures of an $\eta $-Einstein almost paracontact metric  Walker 3-manifold. In \cite{Z2} a $\xi $-sectional curvature and a  $\varphi $-sectional curvature  of a non-degenerate section 
$\alpha $ of an almost paracontact metric manifold are defined:
\begin{itemize}
\item The sectional curvature 
\[K(X,\xi )=\displaystyle\frac{R(X,\xi ,\xi ,X)}{g(X,X)g(\xi ,\xi )-g^2(X,\xi )}\]
of a non-degenerate section 
$\alpha =\rm span \{X,\xi \}$, where the vector field $X$ is orthogonal to $\xi $, is called {\it a $\xi $-sectional curvature} of $\alpha $.
\item 
The sectional curvature 
\[K(X,\varphi X)=\displaystyle\frac{R(X,\varphi X,\varphi X,X)}{g(X,X)g(\varphi X,\varphi X)-g^2(X,\varphi X)}\]
of a non-degenerate section $\alpha =\rm span \{X,\varphi X\}$, where the vector field $X$ is orthogonal to $\xi $, is called {\it a $\varphi $-sectional curvature} of $\alpha $.
\end{itemize}
\begin{thm}\label{Theorem 5.3}
Let $(M,\varphi ,\xi ,\eta ,g)$ be an $\eta $-Einstein almost paracontact metric  Walker 3-manifold. Then:
\par
(i) $M$ is of zero $\xi $-sectional curvature;
\par
(ii) The scalar curvature of $M$ is non-zero. Moreover, $M$ is of  constant  scalar curvature $C$ if and only if $f_{xx}=C$ is  constant;
\par
(iii) Under the equivalent conditions of (ii), $a$, $b$ are constants and $a=-b=\frac{1}{2}C\neq 0$;
\par
(iv) $M$ is of constant non-zero $\varphi $-sectional curvature
if and only if $M$ is of constant non-zero scalar curvature;
\par
(v) Under the equivalent conditions of (ii), $M$ is either  paracosymplectic, or almost paracosymplectic. Moreover, $M$ is paracosymplectic if and only if $2f_{xyz}+f_xf_{xy}-Cf_y=0$
and $M$ is almost paracosymplectic if and only if $2f_{xyz}+f_xf_{xy}-Cf_y\neq 0$.
\end{thm}
\begin{proof}
(i) From (v) in Theorem \ref{Theorem 5.2} it follows that $R(X,\xi ,\xi ,X)=0$, which implies $K(X,\xi )=0$.
\par
(ii) follows from \eqref{5.33} and \eqref{5.6}. 
\par
(iii) is a consequence of \eqref{5.8}.
\par
(iv) Since $M$ is $\eta $-Einstein, from Theorem \ref{Theorem 5.1} it follows that the Ricci operator $Q$ of $M$ is of Segre type $\{11,1\}$ degenerate case, whose eigenvectors are given by \eqref{5.5}. Moreover, $\xi =\pm N$. Then, by using $QN=0$, $QV_1=\frac{1}{2}f_{xx}V_1$ and $QV_2=\frac{1}{2}f_{xx}V_2$, we obtain $g(\xi ,V_1)=g(\xi ,V_2)=0$. Taking into account that $V_1$ and $V_2$ are linearly independent, we have ${\mathbb D}={\rm span}\{V_1,V_2\}$. In a 3-dimensional almost paracontact metric manifold there exist only one section $\alpha =\rm span \{X,\varphi X\}$, where $X\bot \xi $. Thus, $\alpha ={\mathbb D}$. Then $K(X,\varphi X)=K(V_1,V_2)$. By using \eqref{5.1} and \eqref{5.5} we get $R(V_1,V_2,V_2,V_1)=\frac{1}{2}f_{xx}$, $g(V_1,V_1)=0$, $g(V_1,V_2)=1$. Hence, $K(V_1,V_2)=-\frac{1}{2}f_{xx}$, which together with (ii) complete the proof.
\par
(v) By differentiating consecutively with respect to $x$ and $y$ the last equality of   \eqref{5.6},  we obtain $f_{yyx}=0$. From  \eqref{5.6}, by straightforward calculations,  we get $(\xi _1)_x=(\xi _1)_y=0$. From Theorem \ref{Theorem 4.4} and Theorem \ref{Theorem 4.5}, we complete the proof.
\end{proof}
\section{Global almost paracontact metric structures on Walker 3-manifolds}
In literature,  the tensor and vector fields constructed on Walker manifolds are usually local, as their definitions rely  on local coordinates and the local frame field, (e. g.  \cite{GC1},  \cite{GC2} and \cite{CRV}). Our above constructions were also local, but 
our concern now is about the globality of almost paracontact metric structures on Walker 3-manifolds. We provide here three such global almost paracontact metric structures on Walker 3-manifolds. 
\begin{example}\label{Example 6.1} 
This first example is taken from \cite{ZN} in which an almost paracontact metric structure $(\varphi ,\xi ,\eta ,g)$ was defined (see pp. 6, equations (6.53)-(6.54)) on a connected 3-dimensional Lie group $L$, using a basis $\{E_1,E_2,E_3\}$ of left-invariant vector fields: 
\[
\begin{array}{llll}
\varphi E_1=E_2 , \quad \varphi E_2=E_1 , \quad \varphi E_3=0\quad \xi =E_3, \\
\eta (E_3)=1 , \quad \eta (E_1)=\eta (E_2)=0 , \\
g(E_1,E_1)=g(E_3,E_3)=-g(E_2,E_2)=1 ,\\
\quad g(E_i,E_j)=0, \quad i\neq j \in \{1,2,3\}.
\end{array}
\]
In that paper (see pp. 21, equation (7.77)), a particular family of Lie algebras is determined by
\begin{equation*}
[E_1,E_2]=\alpha E_1+\beta E_2, \quad [E_1,E_3]=-\beta E_3, \quad [E_2,E_3]=\alpha E_3,
\end{equation*}
where $\alpha \neq 0$ and $\beta \neq 0$.
The Levi-Civita connection $\nabla $ is given by
\[
\begin{array}{llll}
\nabla_{E_1}E_1=\alpha E_2 , \quad \nabla_{E_1}E_2=\alpha E_1 , \quad \nabla_{E_1}E_3=0 , \\
\nabla_{E_2}E_1=-\beta E_2 , \quad \nabla_{E_2}E_2=-\beta E_1 , \quad \nabla_{E_2}E_3=0  , \\
\nabla_{E_3}E_1=\beta E_3 , \quad \nabla_{E_3}E_2=-\alpha E_3 , \quad \nabla_{E_3}E_3=-\beta E_1-\alpha E_2.   \\
\end{array}
\]
Because the frame and the structure are left-invariant, they are defined on the whole Lie group and therefore their globality follows from left translation from the identity. 
\par
When $\alpha =\beta $, then $D={\rm span}\{E_1+E_2\}$ is a non-zero globally defined lightlike parallel distribution, which shows that 
$(L,g)$ is a Walker manifold.
\end{example}
\begin{example}\label{Example 6.2} 
The second example for an almost paracontact metric structure on Walker 3-manifold is as follows.
Let $M$ be an open subset of ${\R}^3$. Take the global coordinates $(t,u,v)\in M$ and define the metric 
\begin{equation*}
g=
\left(\begin {matrix}0 & e^{-t} &  0 \cr  e^{-t}  & 0  & 0   \cr 0 & 0 &  e^{2t}  
\end {matrix}\right) ,
\end{equation*}
whose signature is (2,1).\\
We construct the Reeb vector field 
\[
\xi =e^{-t}\partial _v, \quad \eta =e^t{\rm d}v,
\]
and the $(1,1)$-tensor field $\varphi $ by
\[
\varphi (\partial _t)=-\partial _t, \quad \varphi (\partial _u)=\partial _u, \quad \varphi (\partial _v)=0,
\]
which are smooth on $M$.  
One can check directly that the structure $(\varphi ,\xi ,\eta ,g)$ is an almost paracontact metric global structure  on $M$, as it  satisfies  \eqref{2.1}, \eqref{2.2} and \eqref{2.5}. \\
We notice that the global vector field $\partial _u$ is everywhere nonzero and lightlike, i.e. $g(\partial _u,\partial _u)=0$. A direct calculations of the Levi-Civita connection $\nabla $ gives only the following nonzero Christoffel symbols in the coordinates $(t,u,v)$:
\[
\Gamma _{tt}^t=-1, \quad \Gamma _{vv}^u=e^{-3t}, \quad \Gamma _{tv}^v=\Gamma _{vt}^v=1.
\]
In particular, $\Gamma _{iu}^k=0$ for every $i, k$, so $\nabla \partial _u=0$. Consequently, $D={\rm span}\{\partial _u\}$ is a globally defined  lightlike parallel distribution, which proves that $(M,g)$ is  an  almost paracontact metric Walker 3-manifold. 
Moreover, the manifold $(M,g)$ is also nonflat, since the curvature tensor field $R$, defined by
$R(X,Y)Z=\nabla _X\nabla _YZ-\nabla _Y\nabla _XZ-\nabla _{[X,Y]}Z$,  yields
$R(\partial _t,\partial _v)\partial _t=2\partial _v\neq 0$.
\end{example}
\begin{rem}\label{Remark 6.1}
In the above examples of this section, the matrix of the Lorentzian metric on a Walker 3-manifold was not expressed in canonical local  coordinates, as in \eqref{3.1}.
\end{rem}
\begin{example}\label{Example 6.3}  Let $(M,h)$ be a simply connected Walker 3-manifold. This topological assumption guarantees that $M$ is orientable and hence parallelizable. Let $\{X,Y,Z\}$ denote a global frame on it, such that $D={\rm span}\{X\}$ and $Y$, $Z$ are some global vector fields, taken to complete the frame.  
With respect to this global frame,  we define a global almost paracontact  structure $(\varphi ,\xi ,\eta )$ by \eqref{3.18} and \eqref{3.5}, where $\xi =(\xi_1,\xi _2,\xi _3)$ and $ \eta =(\xi _3,\xi _2,\xi _1+f\xi _3)$. Based on \cite{Z}, there exist an adapted metric $g$, globally defined on $M$, i.e. $g$ satisfies \eqref{2.5}. Nevertheless, $g$ and $h$ do not always coincide.
\end{example}









\end{document}